\documentclass[a4paper, reqno, 11pt]{amsart}
\pdfoutput=1
\usepackage[utf8]{inputenc}
\usepackage[T1]{fontenc}
\usepackage{amssymb,mathtools}
\usepackage{mathrsfs}
\usepackage{amstext}
\usepackage{amsfonts}
\usepackage{latexsym}
\usepackage{fancyref}
\usepackage[final]{pdfpages}
\usepackage{fancyhdr}
\usepackage{xr} 
\usepackage{textcomp}
\usepackage{subfiles}
\usepackage{amsthm}
\usepackage{bbm}
\usepackage{dsfont}
\usepackage{tabularx}
\usepackage{setspace}
\usepackage{lmodern}
\usepackage{longtable}
\usepackage{geometry}
\usepackage[all]{xy}
\usepackage{microtype}
\usepackage{afterpage}
\usepackage[normalem]{ulem}

\usepackage[all]{xy}
\usepackage{microtype}

\usepackage{enumitem}
\setlist[enumerate,1]{label=\textup{(\arabic*)}}

\usepackage[lite]{amsrefs}
\newcommand*{\MRref}[2]{ \href{http://www.ams.org/mathscinet-getitem?mr=#1}{MR \textbf{#1}}}

\renewcommand*{\PrintDOI}[1]{\href{http://dx.doi.org/\detokenize{#1}}{doi: \detokenize{#1}}}

\BibSpec{book}{%
  +{}  {\PrintPrimary}                {transition}
  +{,} { \textit}                     {title}
  +{.} { }                            {part}
  +{:} { \textit}                     {subtitle}
  +{,} { \PrintEdition}               {edition}
  +{}  { \PrintEditorsB}              {editor}
  +{,} { \PrintTranslatorsC}          {translator}
  +{,} { \PrintContributions}         {contribution}
  +{,} { }                            {series}
  +{,} { \voltext}                    {volume}
  +{,} { }                            {publisher}
  +{,} { }                            {organization}
  +{,} { }                            {address}
  +{,} { \PrintDateB}                 {date}
  +{,} { }                            {status}
  +{}  { \parenthesize}               {language}
  +{}  { \PrintTranslation}           {translation}
  +{;} { \PrintReprint}               {reprint}
  +{.} { }                            {note}
  +{.} {}                             {transition}
  +{} { \PrintDOI}                   {doi}
  +{} { available at \url}            {eprint}
  +{}  {\SentenceSpace \PrintReviews} {review}
}

\usepackage[pdftitle={Toeplitz noncommutative tori},
pdfauthor={ALRS},
pdfsubject={Mathematics}
]{hyperref}

\newtheorem{thm}{Theorem}[section]

\newtheorem{proposition}[thm]{Proposition}
\newtheorem{lemma}[thm]{Lemma}
\newtheorem{corollary}[thm]{Corollary}
\theoremstyle{definition}
\newtheorem{definition}[thm]{Definition}

\theoremstyle{remark}
\newtheorem{remark}[thm]{Remark}
\newtheorem{exm}[thm]{Example}

\newcommand{\thmref}[1]{Theorem~\ref{#1}}
\newcommand{\secref}[1]{Section~\ref{#1}}
\newcommand{\proref}[1]{Proposition~\ref{#1}}
\newcommand{\lemref}[1]{Lemma~\ref{#1}}
\newcommand{\corref}[1]{Corollary~\ref{#1}}

\newcommand{\defref}[1]{Definition~\ref{#1}}

\numberwithin{equation}{section}

\newcommand{\CC}{\mathbb{C}}
\newcommand{\NN}{\mathbb{N}}

\newcommand{\RR}{\mathbb{R}}

\newcommand{\TT}{\mathbb{T}}
\newcommand{\Hilm}[1][E]{\mathcal{#1}}
\newcommand{\ZZ}{\mathbb{Z}}

\newcommand*{\nb}{\nobreakdash}
\newcommand{\BB}{\mathcal{B}}
\def\A{\mathcal A}
\def\1{\mathbbm 1}
\def\B{\mathcal B}

\newcommand{\Tt}{\mathcal{T}}

\newcommand*{\Star}{\(^*\)\nobreakdash-}
\newcommand*{\Comp}{\mathbb K}

\newcommand{\Cst}{\mathrm{C}^*}

\newcommand{\Toepr}{\mathcal{T}_\lambda}

\newcommand{\clsp}{\operatorname{\overline{\mbox{span}}}}

\DeclarePairedDelimiterX{\braket}[2]{\langle}{\rangle}{#1\,\delimsize\vert\,\mathopen{}#2}
\DeclarePairedDelimiterX{\BRAKET}[2]{\langle}{\rangle}{\!\delimsize\langle#1\,\delimsize\vert\,\mathopen{}#2\delimsize\rangle\!}

\def\kmsb{$\mathrm{KMS}_{\beta}$ }
\def\kmsi{$\mathrm{KMS}_{\infty}$ }
\def\kmso{$\mathrm{KMS}_{0}$ }
\def\kmsop{$\mathrm{KMS}_{0^+}$ }

\def\clsp{\overline{\operatorname {span}}}

\def\inv{^{-1}}
\def\Thetad{\Theta_d}
\def\rank{\operatorname{rank}}
\def\Esig{\Hilm^\sigma}

\def\bp{\begin{proof}}
\def\ep{\end{proof}}

\def\ls{L^{\sigma}}
\def\athd{\mathcal A_{\Thetad}}
\def\too{\longrightarrow}

\author{Zahra Afsar}
\address[Z. Afsar]{
School of Mathematics and Statistics\\
Sydney University\\
NSW 2006\\
Australia}
\email{zahra.afsar@sydney.edu.au}

\author{Marcelo Laca}
\address[M. Laca]{Department of Mathematics and Statistics, University of Victoria, Victoria, P.O Box 1700 STN CSC, BC V8W 2Y2, Canada}
\thanks{M. Laca was supported by  the Natural Sciences and Engineering Research Council of Canada, 
Discovery Grant RGPIN-2017-04052}
\email{laca@uvic.ca}

\author{Jacqui Ramagge}
\address[J. Ramagge]{Department of Mathematical Sciences, Durham University, UK}
\email{jacqui.ramagge@durham.ac.uk}

\author{Camila F. Sehnem}
\address[C. F. Sehnem]{School of Mathematics and Statistics, Victoria University of Wellington, P.O. Box 600, Wellington 6140, New Zealand}
\thanks{C.F. Sehnem was supported by the Marsden Fund of the Royal Society of New Zealand, grant No. \!\!18-VUW-056}

\email{camila.sehnem@vuw.ac.nz}

\thanks{This research was supported by the Australian Research Council, Discovery Grant DP200100155.}
\date{24 July 2021, minor changes 11 February 2022}

\title[Equilibrium on extensions of noncommutative tori]{Equilibrium on Toeplitz extensions of higher dimensional noncommutative tori\\  \ \ \\ }

\begin{document}

\begin{abstract}
The $\Cst$-algebra generated by the left-regular representation of $\NN^n$ twisted by a $2$\nb-cocycle is a Toeplitz extension of an $n$\nb-dimensional noncommutative torus, on which each vector $r \in [0,\infty)^n$ determines a one-parameter subgroup of the gauge action. We show that the equilibrium states of the resulting $\Cst$-dynamical system are parametrised by tracial states of the noncommutative torus corresponding to the restriction of the cocycle to the vanishing coordinates of $r$. These in turn correspond to probability measures on a classical torus whose dimension depends on a certain degeneracy index of the restricted cocycle. Our results generalise the phase transition on the Toeplitz noncommutative tori used as building blocks  in recent work of Brownlowe, Hawkins and Sims, and of  Afsar, an Huef, Raeburn and Sims.
\end{abstract}

\maketitle

\section{Introduction}
Suppose $\Theta = (\theta_{i,j})$ is an $n \times n$ antisymmetric matrix with real coefficients. The
  \emph{$n$\nb-dimensional noncommutative torus}~$\A_\Theta$ is the universal $\Cst$\nb-algebra
generated by unitaries $U_1,\ldots , U_n$ satisfying 
\[U_{j}U_{k}={e^{-2\pi i\theta_{j,k}}}U_{k}U_{j} \qquad j, k = 1, 2, \ldots , n.\]   
The matrix $\Theta$ determines a circle-valued $2$\nb-cocycle 
$\sigma_\Theta \colon \ZZ^n \times \ZZ^n \to \TT$ on $\ZZ^n$ given by  
\begin{equation}\label{eq:cocycle-theta}
\sigma_\Theta (x,y)\coloneqq e^{-\pi i  \braket{x}{ \Theta  y}}, 
\end{equation} 
and the noncommutative torus $\A_\Theta$ is canonically isomorphic to the twisted group $\Cst$\nb-algebra $\Cst(\ZZ^n,\sigma_{\Theta})$, see \cite{C.Phillips, MR1047281}. That is, $\A_\Theta$ is also the universal $\Cst$\nb-algebra for unitary 
$\sigma_\Theta$\nb-representations of~$\ZZ^n$.  In particular, a concrete realisation of $\A_\Theta$ as a $\Cst$\nb-algebra of operators on a Hilbert space is provided by  the usual left regular unitary $\sigma_\Theta$\nb-representation $\lambda^{\sigma_\Theta}$  of~$\ZZ^n$ on~$\ell^2(\ZZ^n)$. This is defined on a canonical orthonormal basis vector by $\lambda^{\sigma_\Theta}_x \delta_y = e^{-\pi i \braket{x}{ \Theta  y}}\delta_{x+y}$ for $x,y \in \ZZ^n$. The universal property defining $\A_\Theta$ holds for $\lambda^{\sigma_\Theta}$ because $\ZZ^n$ is abelian, hence amenable. See, for example, \cite{10.2307/1995410} for the theory of twisted group algebras.

In a recent paper, Latr\`emoli\'ere and Packer \cite{LP} considered noncommutative solenoids, which  can be described as inductive limits of $2$-dimensional noncommutative tori, in analogy with their commutative counterparts. Subsequently, Brownlowe, Hawkins and Sims \cite{brownlowe_hawkins_sims_2019} defined Toeplitz extensions of the noncommutative solenoids of \cite{LP} and studied their phase transitions under  natural dynamics. The solenoidal extensions of \cite{brownlowe_hawkins_sims_2019}  are constructed as direct limits of building blocks, each of which is a Toeplitz-type extension of a noncommutative $2$\nb-torus. These Toeplitz-type extensions are obtained by replacing one of the two canonical unitary generators in the presentation of a noncommutative $2$\nb-torus by an isometry. 

Prompted by the constructions given in \cite{LP,brownlowe_hawkins_sims_2019}, Afsar, an Huef, Raeburn and Sims \cite{AaHRS} defined  higher-rank versions of Toeplitz noncommutative solenoids, using  Toeplitz-type extensions of higher-dimensional noncommutative tori as building blocks.  As in \cite{brownlowe_hawkins_sims_2019}, each of these building blocks is obtained by replacing  one of the two generating unitaries of a noncommutative $2$\nb-torus by a Nica-covariant isometric representation~$V$ of $\NN^k$, and the other  by a unitary representation $U$ of $\ZZ^d$; the commutation relations satisfied by $U$ and $V$ are then encoded in a 
$k \times d$ real matrix~$\theta$, see \cite[Equation (2.1)]{AaHRS}.  The phase transition computed in \cite{AaHRS} is  parametrised by measures on the classical torus $\TT^d$, which is the spectrum of the subalgebra $\Cst(U_x\mid x\in \ZZ^d)$.

The present investigation was sparked by a question raised by Ian Putnam about the possibility of replacing the classical torus 
$\Cst(U_x\mid x\in \ZZ^d)$ by a noncommutative one.  Instead of viewing a higher-rank noncommutative torus as a $\Cst$-algebra
generated by two unitary representations, one of $\ZZ^k$  and one of $\ZZ^d$,  satisfying a commutation relation, we consider  all the $n=k+d$ generators on the same footing. Thus, we regard an $n$\nb-dimensional noncommutative torus as the
 twisted group algebra  $\Cst (\ZZ^n,\sigma)$ in which the twist is given by a circle-valued $2$-cocycle~$\sigma$.  Since the circle-valued second cohomology of $\NN^n$ is the same as that of $\ZZ^n$ by \cite[Corollary 2.3]{LR1}, it seems natural to take the {\em reduced twisted semigroup $\Cst$-algebra} $\Tt_r(\NN^n,\sigma)$ as the appropriate Toeplitz extension of the noncommutative torus associated to an extension of~$\sigma$ to $\ZZ^n$. The Toeplitz-type extensions of noncommutative tori from \cite{AaHRS} can then be obtained as quotients of special cases of ours. 

Our purpose here is two-fold. Firstly we aim to describe various realisations of our Toeplitz noncommutative tori,  using product systems, semigroup crossed products, and generators and relations. Secondly, we aim to compute their phase transitions of  \kmsb states with respect to one-parameter subgroups of the canonical dual action of $\TT^n$. Along the way we also provide an explicit description of the center and the space of traces of higher-rank noncommutative tori. 
We do not address here the question of how to use our Toeplitz noncommutative tori   as building blocks for
generalisations of the higher-rank noncommutative solenoids defined in~\cite{AaHRS},  leaving the issue for further work.

We begin in \secref{ToepExtNoncommTori} with a brief discussion of $2$\nb-cocycles on $\NN^n$ and we introduce their twisted semigroup algebras using as a model the left regular representation twisted by a cocycle. Then we show that these twisted semigroup algebras have a universal property with respect to twisted isometric representations  that satisfy certain $^*$\nb-commuting relations, which we show to be equivalent to Nica-covariance. They can also be characterised as Nica--Toeplitz $\Cst$\nb-algebras of product systems over $\NN^n$ and as twisted semigroup crossed products.  We then use the cocycle extension results of~\cite{LR1} to show that these twisted semigroup algebras are indeed extensions of rank\nb-$n$ noncommutative tori.

As with every discrete abelian group, the circle-valued second cohomology of $\ZZ^n$ can be described in terms of symplectic bicharacters, which in this case are parametrised by antisymmetric  matrices over $\RR$; with our conventions these matrices are determined up to even integers. So we assume that $\sigma=\sigma_\Theta$ comes from an antisymmetric matrix $\Theta$ as in \eqref{eq:cocycle-theta}, which allows us to give a presentation of $\Tt_r(\NN^n,\sigma_\Theta)$ in terms of generators and relations that are entirely analogous to those in the original presentation of  rotation algebras and, more generally, of higher-rank noncommutative tori.  Using this presentation we  show in \proref{pro:AaHRStoeptori} that the building blocks from \cite{AaHRS} are quotients of  particular cases of our Toeplitz extensions of noncommutative tori, corresponding to antisymmetric matrices that have two diagonal blocks of zeros and such that the generating isometries corresponding to the second block are unitaries.

In \secref{sec:characterisation} we consider the dynamics induced on $\Tt_r(\NN^n,\sigma_\Theta)$ by a vector $r\in \RR^n$ with $k$ positive coordinates and $d = n-k$ vanishing coordinates. In order to compute the \kmsb states, 
we follow the strategy developed in \cite{LACA1998330} for the Bost--Connes system, showing that \kmsb states are parametrised by tracial states of a corner. The projection $Q$ that defines the corner is the common orthogonal complement of the range projections of the generating isometries that are not fixed by the dynamics, that is, those corresponding to the nonvanishing coordinates of the vector $r$.

In \secref{sec:traces} we show that the corner $Q\Tt_r(\NN^n,\sigma_\Theta)Q$ is actually isomorphic to the Toeplitz noncommutative torus associated to the $d\times d$ submatrix $\Theta_d$ corresponding to the vanishing coordinates of the vector~$r$. 
One of our main technical results is \proref{pro:KMSfromtraces}, where we give a formula for \kmsb states in terms of tracial states of $\Tt_r(\NN^d,\sigma_{\Theta_d})$ and a finite Euler product that depends on the cocycle. The tracial states of this rank-$d$ Toeplitz noncommutative torus factor through the associated noncommutative torus. This led us to compute the tracial states of the higher-rank noncommutative torus $\A_D$ associated to a $d\times d$ antisymmetric real matrix~$D$. We show that these tracial states come from states of the center $\mathrm{Z}(\A_D)$, which is a classical torus of dimension equal to a certain degeneracy index of $D$. 

Our  main results are in \secref{sec:mainresults}, starting with \thmref{thm:main}, where we give a parametrisation of the simplex of \kmsb states in terms of probability measures on a classical torus. As a corollary, for the special case of antisymmetric matrix with zero diagonal blocks, we recover the phase transition for the building blocks of \cite{AaHRS}.

In  \secref{sec:kmsiandkmso} we compute the equilibrium states at $\beta =\infty$ and $\beta =0$. In addition to the usual notions of ground states and invariant traces, we also study the limiting cases, namely $\beta\to \infty$, which gives the  \kmsi states, and the limits of \kmsb states as $\beta \to 0^+$, a class that we call \kmsop states. We show that  in general not every invariant trace is such a \kmsop state.

Finally,  in \secref{sec:alternative} we realise the Toeplitz noncommutative tori $\Tt_r(\NN^n,\sigma_\Theta)$ as the Nica--Toeplitz algebra of a product system over 
$\NN^k$ in which the coefficient algebra is the Toeplitz noncommutative torus associated to the restriction of $\sigma_\Theta$ to $\NN^d$. This allows us to compare our results  to those obtained from the characterisation of \kmsb states given in \cite{ALN20}.

\subsection*{Acknowledgments:} This research was started during a visit of Z.~Afsar, J.~Ramagge and C.~F.~Sehnem to the University of Victoria and continued through visits of C.~F.~Sehnem and M.~Laca to the University of Sydney. We would like to acknowledge the support of these institutions and to thank them for the hospitality.

\section{Toeplitz extensions of noncommutative tori}\label{ToepExtNoncommTori}
We wish to consider the natural Toeplitz-type extensions of noncommutative tori. Since there are some
nuances associated with passing from unitaries to isometries, it is more straightforward to introduce these extensions spatially, via concrete projective representations of $\NN^n$. We begin with a very brief introduction to projective isometric representations of semigroups  and twisted semigroup $\Cst$\nb-algebras.
\subsection{Projective isometric representations}
\begin{definition} Suppose $P$ is a subsemigroup of a group $G$.
A circle-valued semigroup \emph{$2$\nb-cocycle} is a function $\sigma\colon P\times P \to \TT$ such that
\[
\sigma(p,q) \sigma(pq,r) = \sigma(p,qr) \sigma(q,r), \qquad p,q,r \in P.
\]
\end{definition}
The set of all $\TT$-valued $2$\nb-cocycles on~$P$ is a group under pointwise multiplication, which we denote by $Z^2( P,\TT)$. 

\begin{definition} Let $\sigma \in Z^2(P,\TT)$ be a $2$\nb-cocycle on~$P$. An {\em isometric $\sigma$-representation} of~$P$ 
is a map $p \mapsto V_p$ into the semigroup of isometries on a Hilbert space~$\Hilm[H]$ such that 
\[
 V_p V_q = \sigma(p,q) V_{p+q},  \qquad \quad p,q\in P.
\]
\end{definition}

The isometric $\sigma$\nb-representations are multiplicative only up to scalars. If~$P$ is unital, then $V_e$ is a multiple of the identity on~$\Hilm[H]$, where~$e$ denotes the unit element of~$P$. When $\sigma$ is normalised, that is, $\sigma(e,e)=1$,  then $V_e=1_{\Hilm[H]}$. In this case $\sigma(p,e)=\sigma(e,p)=1$ for all~$p\in P$.  

One can generate an example of a $2$\nb-cocycle by starting with an arbitrary function $\lambda\colon P \to \TT$ 
and defining $\sigma_\lambda(p,q) \coloneqq \lambda(pq)\lambda(p)\inv \lambda(q)\inv$, for all $p,q\in P$. 
This is called a {\em coboundary}, and is trivial in the sense that any $\sigma_\lambda$\nb-representation can be
transformed into a true isometric representation through multiplication by the scalar-valued function~$\lambda$. 
 We denote by $B^2(P,\TT)$ the subgroup of all coboundaries. 
A standard argument shows that the representations associated to two $2$\nb-cocycles that differ by a coboundary are equivalent, in the sense that one is a multiple of the other by a circle-valued function. So what is interesting is the group $$H^2(P,\TT)\coloneqq Z^2(P,\TT)/B^2(P,\TT)$$ of cocycles modulo coboundaries, which is called the \emph{second cohomology group} of~$P$. 

Next we verify that for each $\sigma$ in $Z^2(P,\TT)$, there exists an isometric $\sigma$\nb-representation of~$P$.

\begin{proposition} Suppose~$\sigma$ is a circle-valued $2$\nb-cocycle on~$P$ and let $\{\delta_q\mid q\in P\}$ be the canonical orthonormal basis of $\ell^2(P)$. Then for each $p\in P$,
 the map $\ls_p$ defined by 
\[\ls_p\delta_q\coloneqq\sigma(p,q)\delta_{pq}, \qquad\qquad (q\in P)\]
extends uniquely to an isometry on $\ell^2(P)$ such that
\begin{equation}\label{eq:left-inverse}
\begin{aligned}
(L^{\sigma}_p)^*\delta_q=\begin{cases}\overline{\sigma(p,q)}\delta_{p\inv q}&\text{if }p\leq q,\\
0  &\text{otherwise}.
\end{cases}
\end{aligned}
\end{equation} Moreover, we have for all $p,q\in P$
$$\ls_p \ls_q = \sigma(p,q) \ls_{pq},\qquad\text{and} \qquad \ls_p (\ls_p)^* = 1_{pP},$$
where $1_{pP}$ is the multiplication operator by the characteristic function of the set $pP=\{pq\mid q\in P\}$.
\end{proposition}
\bp
Since $P$ is a cancellative semigroup,  $\ls_p$ maps the elements of the standard orthonormal basis one to one and onto those of $\{\delta_{pq}\mid q\in P\}$, multiplied by complex numbers of modulus~$1$. So $\ls_p$ extends uniquely by linearity and continuity 
to an isometric operator (also denoted $\ls_p$) on $\ell^2(P)$ with range $\ell^2(pP)$.
The second assertion follows because the formula given in the right-hand side of \eqref{eq:left-inverse} defines a left inverse for~$\ls_p$ on orthonormal basis vectors since $\sigma(p,q) \overline{\sigma(p,q)} = 1$.

To prove the last claim, we compute first
 \[
 \ls_p \ls_q \delta_r =  \ls_p \sigma(q,r) \delta_{qr} = \sigma(p,qr)\sigma(q,r) \delta_{pqr},
 \]
 and then 
 \[
\ls_{pq}\delta_r = \sigma(pq,r) \delta_{pqr}.
\]
So $\sigma(p,q) \ls_{pq}\delta_r = \sigma(p,q)\sigma(pq,r) \delta_{pqr} = \ls_p \ls_q \delta_r $ by the cocycle identity.
\ep
\begin{definition} 
Let $\sigma$ be a circle-valued $2$\nb-cocycle on~$P$. The  {\em left regular $\sigma$-representation} of $P$ is the map $p\mapsto \ls_p \in \ell^2(P)$.
The {\em reduced semigroup $\Cst$\nb-algebra twisted by $\sigma$} is the $\Cst$\nb-subalgebra of $\BB(\ell^2(P))$ generated by the operators $\ls_p$ with $p\in P$; it will be denoted by~$\Tt_r(P,\sigma)$. 
\end{definition}
When the cocycle is understood and there is no risk of confusion, we simply refer to $\Tt_r(P,\sigma)$ as a (reduced) twisted semigroup $\Cst$\nb-algebra.

\subsection{Toeplitz extensions of higher dimensional noncommutative tori} 
 By analogy with the realisation of an $n$\nb-dimensional noncommutative torus as a twisted group algebra of $\ZZ^n$, we define the Toeplitz noncommutative torus as a twisted semigroup $\Cst$\nb-algebra associated to a $2$\nb-cocycle on the semigroup $\NN^n$. We will see below in \proref{pro:toeplitz-extension} that this is indeed a natural extension of the corresponding noncommutative torus.

\begin{definition} Let $\sigma$ be a circle-valued $2$\nb-cocycle on $\NN^n$. The associated $n$\nb-dimensional \emph{Toeplitz noncommutative torus} is the twisted semigroup $\Cst$\nb-algebra $\Tt_r(\NN^n,\sigma)$.
\end{definition}

The Toeplitz noncommutative torus $\Tt_r(\NN^n,\sigma)$  also has a universal property with respect to a certain class of twisted isometric representations. We define this next by adding  extra relations modelled on  the left regular $\sigma$\nb-representation.

\begin{definition}\label{dfn-covariance}
A {\em covariant  isometric $\sigma$\nb-representation} of $\NN^n$ on a Hilbert space $\mathcal H$ 
is a map $w\colon \NN^n \to  \mathcal B(\mathcal H)$ such that
\begin{enumerate}
\item\label{dfn:cov-relation1} $w^*_pw_p=1$;
\item \label{dfn:cov-relation2} $w_p w_q=\sigma_{}(p,q)w_{p+q}$; 
\item\label{dfn:cov-relation3} $w^*_pw_q=\overline{\sigma_{}(p,-p+(p\vee q))}\sigma(q,-q+(p\vee q)) w_{-p+(p\vee q)}w_{-q+(p\vee q)}^*$.
\end{enumerate}
\end{definition}
 The third condition above may seem a bit unusual at first but it will be convenient in some computations. We show next that it is equivalent to the better-known Nica-covariance condition from \cite[Section 3.3]{Nica:Wiener--hopf_operators}.
\begin{lemma} \label{lem:nicacovariance} An isometric $\sigma$\nb-representation $w$ of $\NN^n$
satisfies condition (3) of \defref{dfn-covariance} if and only if it is Nica-covariant, that is,
\[
w_pw_p^* w_qw_q^* = w_{p\vee q}w_{p\vee q}^* \qquad p,q\in \NN^n.
\]
In particular, 
the left regular $\sigma$\nb-representation of~$\NN^n$ is covariant in the sense of Definition~\ref{dfn-covariance}. 
\begin{proof}
Assume first (3) holds. Then
 \begin{equation*}
w_pw_p^* w_qw_q^*
=\overline{\sigma(p,-p+(p\vee q))}\sigma(q,-q+(p\vee q))w_p w_{-p+(p\vee q)}(w_{-q+(p\vee q)})^*w_q^*
=w_{p\vee q} w_{p\vee q}^*,
\end{equation*}
where we used the cocycle to combine factors in the last equality. Conversely, assume that $w$ is Nica covariant. Then 
\begin{equation*}
\begin{aligned}
w_p^* w_q 
&= w_p^*w_pw_p^* w_qw_q^* w_q = w_p^* (w_{p\vee q} w_{p\vee q}^*)  w_q \\
&= \overline{\sigma(p,-p+(p\vee q))} \sigma(q,-q+(p\vee q)) w_p^* (w_p w_{-p+(p\vee q)} w_{-q+(p\vee q)}^* w_q^*)  w_q\\
&= \overline{\sigma(p,-p+(p\vee q))} \sigma(q,-q+(p\vee q)) w_{-p+(p\vee q)} w_{-q+(p\vee q)}^*  
\end{aligned}\end{equation*}
Since $\ls_p(\ls_p)^* $ is the operator of multiplication by the characteristic function of ${p+\NN^n}$,  the left regular $\sigma$-representation $\ls$ satisfies Nica-covariance, hence also condition (3).
\end{proof}
\end{lemma}

The following immediate consequence of \defref{dfn-covariance}  will be needed in the following sections.
\begin{lemma}\label{lem:relprimecomm}
Suppose $w\colon \NN^n \to  \mathcal B(\mathcal H)$ is a covariant  isometric $\sigma$\nb-representation of $\NN^n$, and suppose $x,y \in \NN^n$ satisfy $x \vee y = x+y$. Then
\[
w_xw_x^* w_y = w_y w_xw_x^*.
\]
\begin{proof}
Using first  property (3) and then twice property (2) of \defref{dfn-covariance}, we obtain
\begin{equation*}
\begin{aligned}
w_xw_x^* w_y =   \overline{\sigma_{}(x,y)}\sigma(y,x) w_xw_{y}w_{x}^* 
&= \overline{\sigma_{}(x,y)}\sigma(y,x)  \sigma(x,y) w_{x+y} w_{x}^*\\& =
\overline{\sigma_{}(x,y)}\sigma(y,x)  \sigma(x,y)\overline{\sigma_{}(y,x)} w_yw_{x}w_{x}^* .
\end{aligned}
\end{equation*}
This completes the proof of the lemma because  $\overline{\sigma_{}(x,y)}\sigma(y,x)  \sigma(x,y)\overline{\sigma_{}(y,x)} =1$.
\end{proof}
\end{lemma}

It will be useful to have a description of $\Tt_r(\NN^n, \sigma)$ as the Nica--Toeplitz algebra of a  compactly aligned  product system over 
$\NN^n$ (see \cite[Theorem~6.3]{Fowler:Product_systems}). This will allow us to apply results of~\cite{FOWLER1998171} and to relate our notion of a Toeplitz noncommutative torus to the one considered by Afsar, an~Huef, Raeburn and Sims \cite{AaHRS}. In order to simplify the notation, we assume from now on that~$\sigma$ is normalized, that is, $\sigma(0,0)=1$. For each $p\in \NN^n$, let $\Hilm_p\coloneqq \CC$ as a complex vector space and we write $\lambda\delta_p$ for an element $\lambda\in \Hilm_p$. We equip $\Hilm_p$ with the canonical structure of a correspondence over~$\CC$, that is, the left and right actions are given by multiplication in $\CC$ and the right inner product is simply $\braket{\lambda\delta_p}{\alpha\delta_p}=\bar{\lambda}\alpha\delta_0$ for all $\lambda,\alpha\in\Hilm_p$. We will use the cocycle $\sigma$ to define the multiplication maps of the product system 
$\Esig=(\Hilm_p)_{p\in\NN^n}$. For $p,q\in \NN^n$, we set \begin{equation*}
\begin{aligned}
\mu_{p,q}&&\colon \Hilm_p\times\Hilm_q&\to \Hilm_{p+q}\\
&&(\lambda\delta_p,\alpha\delta_q)&\mapsto\sigma(p,q) \lambda\alpha\delta_{p+q}.
\end{aligned}
\end{equation*} These maps are associative because $\sigma$ is a cocycle. So $\Esig$ is a product system over $\NN^n$, which is compactly aligned because $\Comp(\Hilm_p)=\CC$ for all $p\in \NN^n$.

\begin{proposition}\label{prop:product_system} Let $\mathcal{N}\Tt_{\Esig}$ be the Nica--Toeplitz algebra of the product system $\Esig=(\Hilm_p)_{p\in\NN^n}$ as above. The map which sends $\delta_p\in\Hilm_p$ to the isometry~$\ls_p$ in the left regular $\sigma$-representation induces an isomorphism $\mathcal{N}\Tt_{\Esig}\cong \Tt_r(\NN^n,\sigma)$.
\end{proposition}

\begin{proof} Let $\phi_p$ denote the map $\lambda\delta_p\mapsto \lambda \ls_p$. For $p,q\in\NN^n$, we have $$\phi_p(\delta_p)\phi_q(\delta_q)=\ls_p\ls_q=\sigma(p,q)\ls_{p+q}=\phi_{p+q}(\delta_p\cdot \delta_q).$$ Hence $\phi=\{\phi_p\}_{p\in P}$ preserves the multiplication on~$\Esig$. It is also compatible with inner products because each~$\ls_p$ is an isometry. Notice next that the  homomorphism $\phi^{(p)}\colon\Comp(\Hilm_p)\to \Tt_r(\NN^n, \sigma)$ induced by~$\phi$ is given by $\lambda1_{\Hilm_p}\in\Comp(\Hilm_p)\mapsto \lambda \ls_p(\ls_p)^*.$ 
Hence \lemref{lem:nicacovariance} gives
\begin{equation*}
\phi^{(p)}(1_{\Hilm_p})\phi^{(q)}(1_{\Hilm_q})
=\ls_p(\ls_p)^*\ls_q(\ls_q)^*
=\ls_{p\vee q}(\ls_{p\vee q})^*
= \phi^{p\vee q}(1_{\Hilm_{p\vee q}}).
\end{equation*} 
This proves that $\phi$ is Nica covariant and therefore it induces a surjective  homomorphism $\mathcal{N}\Tt_{\Esig}\to\Tt_r(\NN^n, \sigma)$, which we still denote by~$\phi$. To see that $\phi$ is an isomorphism, notice that  \[
 \prod_{\substack{i=1}}^k(1_{\ell^2(\NN^n)}-\ls_{p_i}\ls_{p_i}{}^*)\delta_0=\delta_0
 \]
 whenever $p_1,p_2,\ldots,p_k\in \NN^n\setminus\{0\}$. Because $\NN^n$ is abelian, we can apply \cite[Theorem~5.1]{FOWLER1998171} (see  also \cite[Example~5.6]{FOWLER1998171}(1)) to deduce that~$\phi$ is injective.
\end{proof}

\begin{corollary}\label{cor:faithful_charac} Let $\sigma$ be a $2$\nb-cocycle on~$\NN^n$ and let $w\colon \NN^n\to B$ be a covariant isometric $\sigma$\nb-representation in a $\Cst$\nb-algebra $B$. Then there is a  homomorphism $\rho\colon\Tt_r(\NN^n,\sigma)\to B$ that sends the isometry $\ls_p$ to $w_p$ for all $p\in\NN^n$. Moreover, $\rho$ is faithful if and only if
\begin{equation}\label{eq:faithful}
\prod_{\substack{i=1}}^n(1_{\Hilm[H]} - w_{e_i}w_{e_i}^*)\neq 0.
\end{equation} As a consequence, relations \textup{(1)--(3)} from Definition~\textup{\ref{dfn-covariance}} form a presentation of $\Tt_r(\NN^n,\sigma)$.
\end{corollary}
\bp By \lemref{lem:nicacovariance}, an isometric $\sigma$\nb-representation of $\NN^n$ satisfies relation (3) of Definition~\ref{dfn-covariance} if and only if it is Nica covariant. Hence the first assertion in the statement follows from an application of the universal property of $\mathcal{N}\Tt_{\Esig}$ combined with the isomorphism $\Tt_r(\NN^n,\sigma)\cong \mathcal{N}\Tt_{\Esig}$ from Proposition~\ref{prop:product_system}.

Since $\NN^n$ is abelian, it follows from \cite[Theorem~5.1]{FOWLER1998171} that a representation of $\Tt_r(\NN^n,\sigma)$ is faithful if and only if 
\begin{equation}\label{eq:faithfulf}
\prod_{\substack{i=1}}^k(1 - w_{p_i}w_{p_i}^*)\neq 0\quad\text{ whenever }\,p_1,p_2,\ldots,p_k\in \NN^n\setminus\{0\}.
\end{equation} Now every projection of the form $1-w_{p}w_{p}^*$ with $p\neq 0$ dominates a projection $1-w_{e_i}w_{e_i}^*$ for some $i\in\{1,\ldots,n\}$, and so \eqref{eq:faithful} and \eqref{eq:faithfulf} are equivalent. This establishes the equivalence in the second statement. The last assertion then follows because the left regular $\sigma$\nb-representation of~$\NN^n$ in $\Tt_r(\NN^n,\sigma)$ satisfies relations (1)--(3) of \defref{dfn-covariance}.
\ep

The $\Cst$\nb-algebra $\Tt_r(\NN^n, \sigma)$ can also be realised as a twisted semigroup crossed product.  
Following~\cite{LACA1996415} 
we consider the characteristic function $\1_{p+\NN^n}$ of the cone with vertex $p \in \NN^n$. We define~$B_{\NN^n}$ to be the $\Cst$\nb-subalgebra of $\ell^\infty(\NN^n)$ generated by all the projections $\1_{p+\NN^n}$ for~$p \in \NN^n$. There is a canonical action $\beta=\{\beta_p \mid p\in \NN^n\}$ of~$\NN^n$ on~$B_{\NN^n}$ by injective endomorphisms with hereditary range. Precisely, the endomorphism $\beta_p$ is defined on a generator by 
$$\beta_p(\1_{q+\NN^n})\coloneqq\1_{p+q+\NN^n}.$$

\begin{proposition} Let $(B_{\NN^n},\NN^n, \beta)$ be as above and let $\sigma$ be a $2$\nb-cocycle on~$\NN^n$. Then $\Tt_r(\NN^n,\sigma)$ is isomorphic to the twisted crossed product $B_{\NN^n} \rtimes_{\beta,\sigma} \NN^n$ via an isomorphism that identifies the canonical generating isometries.
\bp The result follows from \cite[Theorem~4.3]{FOWLER1998171} since the product system $\Esig$ has one-dimensional fibres. 
\ep
\end{proposition}

In analogy to what happens with the semigroup $\Cst$\nb-algebra $\Toepr(\NN^n)$ (see, for example, \cite{Nica:Wiener--hopf_operators}), there is a canonical  \emph{gauge action} $\gamma$ of~$\TT^n$ on~$\Tt_r(\NN^n,\sigma)$. It is implemented spatially by the unitaries $\{U_z: z\in \TT^n\}$ defined by $U_z \delta_p = z^p \delta_p$ for $p\in\NN^n$, where $z^p=\prod_{\substack{i=1}}^nz_i^{p_i}$. The gauge action $z\mapsto \gamma_z\coloneqq \operatorname{Ad}_{U_z}$  can also be obtained through the relations \ref{dfn:cov-relation1}--\ref{dfn:cov-relation3}
by stating that $\gamma_z$ is the automorphism that sends an isometry~$\ls_p$ to~$z^p\ls_p$. This action gives rise to a faithful conditional expectation $$E\colon \Tt_r(\NN^n,\sigma)\to\Tt_r(\NN^n,\sigma)^{\gamma}=\overline{\mathrm{span}}\{\ls_p(\ls_p)^*\mid p\in\NN^n\}$$ given by $E(b)=\int_{\substack{\TT^n}}\gamma_\lambda(b) d\lambda$
and $\Tt_r(\NN^n,\sigma)$ is the closed linear span of the set $\{\ls_p(\ls_q)^*\colon p,q\in\NN^n\}$. 

The Toeplitz noncommutative torus $\Tt_r(\NN^n,\sigma)$ has a canonical $n$\nb-dimensional noncommutative torus as a quotient. We show this in the next result, where we also  identify the canonical generating set for the kernel of the quotient map.

 \begin{proposition}\label{pro:toeplitz-extension} Every $2$\nb-cocycle $\sigma\in Z^2(\NN^n ,\TT)$  has an extension to a 
$2$\nb-cocycle $\tilde\sigma$ on~$\ZZ^n$, and this extension is unique up to coboundaries. Moreover, the map that sends an isometry $L_p^{\sigma}\in\Tt_r(\NN^n,\sigma)$ to the unitary $\lambda_p^{\tilde\sigma}\in \Cst_r(\ZZ^n,\tilde\sigma)$ determines an exact sequence 
\[
 0 \too \mathcal{I} \too \Tt_r(\NN^n,\sigma) \too \Cst_r(\ZZ^n,\tilde\sigma) \too 0,
\]
where $\mathcal{I}$ is the ideal of $\Tt_r(\NN^n,\sigma)$ generated by the projections $$\{1-\ls_{e_j}(\ls_{e_j})^*\mid j=1,\ldots,n\}.$$ 
 \end{proposition}
 \begin{proof}  Since $\NN^n$ is abelian and  $\ZZ^n = \NN^n - \NN^n$ is the group of right (and left) quotients of $\NN^n$,
 $2$\nb-cocycles on~$\NN^n$ extend  to $2$\nb-cocycles on $\ZZ^n$, uniquely up to coboundaries, by \cite[Corollary~2.3]{LR1}. More precisely, the restriction of $2$\nb-cocycles from $\ZZ^n$ to~$\NN^n$ induces an isomorphism of the second cohomology group $H^2(\ZZ^n,\TT)$ onto $H^2(\NN^n,\TT)$.

Let $\sigma$ be a $2$\nb-cocycle on $\NN^n$ and let $\tilde{\sigma}$ be a $2$\nb-cocycle on $\ZZ^n$ extending~$\sigma$.  It follows that the map that sends $p\in\NN^n$ to the unitary $\lambda_p^{\tilde\sigma}\in \Cst_r(\ZZ^n,\tilde\sigma)$ is a covariant isometric $\sigma$\nb-representation of~$\NN^n$. This gives a  homomorphism $\phi\colon\Tt_r(\NN^n,\sigma)\to \Cst_r(\ZZ^n,\tilde\sigma)$ mapping $\ls_p$ to $\lambda_p^{\tilde\sigma}$, which is surjective because the unitaries $\lambda_p^{\tilde\sigma}$ for $p\in \NN^n$ generate $\Cst_r(\ZZ^n,\tilde\sigma)$ as a $\Cst$\nb-algebra. 

Let $\mathcal{I}$ be the ideal of $\Tt_r(\NN^n,\sigma)$ generated by the projections $1-\ls_{e_j}(\ls_{e_j})^*$ for $j=1,\ldots, n$. Then $\phi$ vanishes on $\mathcal{I}$ since it vanishes on the projection $1-\ls_{e_j}(\ls_{e_j})^*$ for each $j\in\{1,\ldots, n\}$.  Hence it factors through a  homomorphism $$\dot{\phi}\colon \Tt_r(\NN^n,\sigma)/\Hilm[I]\to \Cst_r(\ZZ^n,\tilde\sigma).$$ So it remains to show that $\dot{\phi}$ is injective. To do so, we invoke the description of $\Tt_r(\NN^n,\sigma)$ as the Nica--Toeplitz algebra of the product system $\Esig$ arising from Proposition~\ref{prop:product_system}.                                The quotient map $\dot{\phi}$ gives a representation of $\Esig$ in 
$\Cst_r(\ZZ^n,\tilde\sigma)$. In this case $\Esig$ is a product system of Hilbert bimodules since $\Comp(\Hilm_p)=\CC$ for all $p\in\NN^n$. We claim that the canonical representation of $\Esig$ in the quotient $\Tt_r(\NN^n,\sigma)/\Hilm[I]$ is Cuntz--Pimsner covariant on $\BRAKET{\Hilm_p}{\Hilm_p}=\CC$ for all $p\in \NN^n$, where $\BRAKET{\cdot}{\cdot}$ stands for the left inner product of a Hilbert bimodule. Indeed, it suffices to show that for all $p\in P$ the image of the projection $\ls_p$ under the quotient map is a unitary. We can write $p=\sum_{\substack{j=1}}^np_je_j,$ so that an application of the Nica covariance relation in $\Tt_r(\NN^n,\sigma)$ yields \begin{equation}\label{eq:n-covariance}
\ls_p(\ls_p)^*=\prod_{\substack{j=1}}^n\ls_{p_je_j}(\ls_{p_je_j})^*.
\end{equation} Now each isometry $\ls_{p_je_j}$ has the form $a_j(\ls_{e_j})^{p_j}$ for some scalar $a_j\in \CC$ of modulus~$1$. So the image of $\ls_{p_je_j}$ in $\Tt_r(\NN^n,\sigma)/\Hilm[I]$ is a unitary for all $j=1,2,\ldots,n$. Hence so is the image of $\ls_p$ by \eqref{eq:n-covariance}. This proves the claim. 

Next we observe that $\Esig$ is simplifiable in the sense of  \cite[Definition~3.7]{sehnem2021} since for all $p,q\in\NN^n$ one has $$\BRAKET{\Hilm_p}{\Hilm_p}\BRAKET{\Hilm_q}{\Hilm_q}=\CC=\BRAKET{\Hilm_{p\vee q}}{\Hilm_{p\vee q}}.$$ 
Since the canonical representation of $\Esig$ in the quotient $\Tt_r(\NN^n,\sigma)/\Hilm[I]$ is Cuntz--Pimsner covariant on $\BRAKET{\Hilm_p}{\Hilm_p}$ for all $p\in\NN^n$, it follows from \cite[Lemma~3.11]{sehnem2021} that the fixed-point algebra of the canonical gauge action of $\TT^n$ on $\Tt_r(\NN^n,\sigma)/\Hilm[I]$ is a copy of~$\CC$. The induced conditional expectation onto this copy is faithful because $\ZZ^n$ is amenable. Now the  homomorphism $\dot{\phi}\colon \Tt_r(\NN^n,\sigma)/\Hilm[I]\to \Cst_r(\ZZ^n,\tilde\sigma)$ is gauge-compatible and clearly injective onto the fixed-point algebra~$\CC$, being a nonzero representation.  Hence $\dot{\phi}$ is an isomorphism by \cite[Proposition~19.8]{Exel:Partial_dynamical}, completing the proof of the proposition.
\end{proof}

\begin{remark} It may be helpful  to compare the conclusion of Proposition~\ref{pro:toeplitz-extension} with the analogous result in the theory of (untwisted) semigroup $\Cst$\nb-algebras. When $\sigma$ is the trivial cocycle, that is, $\sigma(p,q)=1$ for all $p,q\in \NN^n$, $\Tt_r(\NN^n,\sigma)$ is simply the Toeplitz $\Cst$\nb-algebra of $\NN^n$. Proposition~\ref{pro:toeplitz-extension} gives a presentation for the (reduced) group $\Cst$\nb-algebra $\Cst_r(\ZZ^n)$. This presentation may be obtained as an application of \cite[Theorem~6.7]{Crisp-Laca} given that $\Cst_r(\ZZ^n)$ is the boundary quotient $\partial\Toepr(\NN^n)$ of $\Toepr(\NN^n)$. 
\end{remark}

\subsection{ Symplectic bicharacters and antisymmetric matrices} 
The second cohomology group of a discrete abelian group~$G$ has a convenient parametrisation in terms of a specific subgroup of cocycles. Let $\sigma^*(x,y) \coloneqq \overline{\sigma(y,x)}$ for $x,y\in G$ and recall that 
a $2$\nb-cocycle~$\sigma$ on~$G$ is a {\em symplectic bicharacter} if it is a bicharacter as a function $\sigma\colon G\times G\to \TT$ and satisfies $\sigma^* = \sigma$. Denote by
 $X^2(G,\TT)$ the group of symplectic bicharacters on~$G$.  
By~\cite[Proposition~3.2]{OPT}, the homomorphism $\sigma\mapsto\sigma\sigma^*$ of  $Z^2( G,\TT)$ to $X^2(G,\TT)$ is surjective and has kernel precisely $B^2(G,\TT)$, so it induces an isomorphism of $H^2(G,\TT)$ onto $X^2(G,\TT)$.

The symplectic bicharacters on~$\ZZ^n$ are precisely the $2$\nb-cocycles associated to $n\times n$ antisymmetric matrices over~$\RR$.  For each $n\times n$ antisymmetric matrix $\Theta=(\theta_{i,j})$ over~$\RR$, 
there is an associated symplectic bicharacter $\sigma_\Theta$ on $\ZZ^n$ given by
\[
 \sigma_\Theta(x,y)= e^{-\pi i \braket{x}{\Theta y}} \qquad\qquad (x,y\in\ZZ^n),
 \]
 and the map $\Theta \mapsto \sigma_\Theta$ is a group homomorphism.
To retrieve the matrix from a given symplectic bicharacter $\sigma$, simply let $e^{\pi i\theta_{j,k}} \coloneqq \sigma(e_k,e_j)$ for $1\leq j\leq k\leq n$, which determines $\theta_{j,k}$ up to an additive {even} integer.

Recall that, at the level of twisted group algebras, $\Cst(\ZZ^n,\sigma_\Theta)$ is canonically isomorphic to the noncommutative torus $\A_\Theta$. Indeed, if $x\mapsto v_x$ is the canonical unitary $\sigma_\Theta$\nb-representation of~$\ZZ^n$ in $\Cst(\ZZ^n,\sigma_\Theta)$,  then the generators $v_{e_j}$ for $j = 1, 2, \cdots, n$ satisfy the defining commutation relations of $\A_\Theta$. Conversely,  if  $U_1, U_2,\ldots , U_n$ are the canonical unitary generators of $\A_\Theta$, and for each  $x=(x_1,x_2,\ldots,x_n)\in\ZZ^n$   we define
\begin{equation}\label{eqn:VintermsofU}
\bar{v}_x \coloneqq  e^{-\pi i \braket{x-x_1e_1}{\Theta x_1e_1 }}e^{-\pi i \braket{x-x_1e_1-x_2e_2}{\Theta x_2e_2}}\ldots e^{-\pi i \braket{x_ne_n}{\Theta x_{n-1} e_{n-1}}}U_{1}^{x_1}U_{2}^{x_2}\ldots U_{n}^{x_n},
\end{equation} 
then $x\mapsto \bar{v}_x$ is a unitary $\sigma_\Theta$-representation of $\ZZ^n$ in $\A_\Theta$.

Next we show that, in analogy to what happens for $\A_\Theta$, the Toeplitz noncommutative torus $\Tt_r(\NN^n,\sigma_\Theta)$ can be characterised as the universal $\Cst$\nb-algebra generated by~$n$ isometries subject to certain commutation relations.
  \begin{proposition}\label{pro:presentationtoeplitztori} 
Let $\Theta$ be an $n \times n$  antisymmetric real matrix and let  $ \sigma_\Theta(p,q)= e^{-\pi i \braket{p}{\Theta q}} $ for $p,q\in\NN^n$ be the associated symplectic bicharacter. Let $\{w_j\mid j\in\{1,\ldots, n\}\}$ be isometries in a $\Cst$\nb-algebra $B$ satisfying the relations
\begin{equation}
\tag{$\mathcal R_\Theta $} 
\begin{cases}w_{j}w_{k}=e^{-2\pi i\theta_{j,k}}w_{k}w_{j} & j, k = 1, 2, \ldots , n; \\
w_{j}^* w_{k}= e^{2\pi i\theta_{j,k}}w_{k}w_{j}^* &  j\neq k.\\
\end{cases}
\end{equation}
Then the map that sends $L^{\sigma_\Theta}_{e_j}$ to $w_j$  induces a homomorphism from $\Tt_r(\NN^n,\sigma_\Theta)$ to~$B$ that maps an isometry  $L^{\sigma_\Theta}_p$ to the product
\begin{equation}\label{eq:ext-formula}
 \sigma_\Theta( p-p_1e_1,e_1)^{p_1}\sigma_\Theta(p-p_2e_2-p_1e_1,e_2)^{p_2}\cdots\sigma_\Theta\left(p-\textstyle{\sum_{\substack{i=1}}^{n-1}}p_ie_i,e_{n-1}\right)^{p_{n-1}}w_{1}^{p_1}w_{2}^{p_2}\ldots w_{n}^{p_n}.
\end{equation}
This homomorphism is an isomorphism if and only if
$\prod_{\substack{j=1}}^n(1 - w_{j}w_{j}^*)\neq 0$.
 \end{proposition}
\bp The first relation in ($\mathcal{R}_\Theta$) implies that the map from $\NN^n$ to $B$ given by \eqref{eq:ext-formula} is an isometric $\sigma_\Theta$\nb-representation of $\NN^n$. The second relation in ($\mathcal{R}_\Theta$) implies that such a representation satisfies condition \ref{dfn:cov-relation3} of  Definition~\ref{dfn-covariance}, and hence is indeed a covariant $\sigma_\Theta$\nb-representation. 
The  homomorphism from $\Tt_r(\NN^n,\sigma_\Theta)$ to~$B$ and the condition for it to be an isomorphism now follow from Corollary~\ref{cor:faithful_charac}. 
\ep

In \cite{AaHRS}, Afsar, an Huef, Raeburn, and Sims  consider certain Toeplitz  extensions of noncommutative tori and use them as building blocks to construct their Toeplitz noncommutative solenoids. We will see below that these building blocks are quotients of special cases of our Toeplitz noncommutative tori.  

Let~$k$ and~$d$ be nonnegative integers with $n = k+d$ and let $\Lambda\in M_{k,d}([0,\infty))$ be a $k\times d$ matrix. The $\Cst$\nb-algebra~$B_{\Lambda}$ considered in~\cite{AaHRS} is the universal $\Cst$\nb-algebra generated by a Nica covariant representation~$v\colon \NN^k\to B_{\Lambda}$ of~$\NN^k$ and a unitary representation $u\colon \ZZ^d\to B_{\Lambda}$ of~$\ZZ^d$, satisfying the commuting relations $$u_xv_p=e^{2\pi i \braket{p}{\Lambda x}}v_pu_x,\qquad x\in\ZZ^d, p\in \NN^k.$$

\begin{proposition}\label{pro:AaHRStoeptori}
Suppose $k$ and $d$ are nonnegative integers with $n = k+d$. For each rectangular $k\times d$ matrix $\Lambda\in M_{k,d}([0,\infty))$ define $\Theta\in M_{k+d}(\RR)$ by
\begin{equation*}\label{LambdafromAaHRS}
\Theta\coloneqq
\left[
\begin{array}{c|c}
0_{k\times k} & \Lambda \\
\hline
-\Lambda^T & 0_{d\times d}
\end{array}
\right].
\end{equation*}
Then the $\Cst$\nb-algebra $B_\Lambda$ associated to $\Lambda$ in \cite{AaHRS} is canonically isomorphic to the quotient of $\Tt_r(\NN^n, \sigma_{\Theta})$ by the ideal generated by the projections $1- L^{\sigma_\Theta}_{e_j}(L^{\sigma_\Theta}_{e_j})^*$ for $j = k+1, k+2, \ldots, k+d$. 
\end{proposition}
\begin{proof} Let $\Hilm[I]_d$ be the ideal of $\Tt_r(\NN^n, \sigma_{\Theta})$ generated by the projections $$\{1- L^{\sigma_\Theta}_{e_j}(L^{\sigma_\Theta}_{e_j})^*\mid j = k+1, k+2, \ldots, k+d\}.$$ For $p\in\NN^n\setminus 0_k\times\NN^d$, we write $\bar{v}_p$ for the image of the isometry~$L^{\sigma_\Theta}_p$ in the quotient $\Tt_r(\NN^n, \sigma_{\Theta})/\Hilm[I]_d$. Then the isometries $\{\bar{v}_p\mid p\in\NN^k\times 0_d\}$ commute among themselves because the upper diagonal block of $\Theta$
consists of zeros. Also, the representation $p\in\NN^k\cong\NN^k\times 0_d\mapsto \bar{v}_p$ is Nica covariant by Proposition~\ref{prop:product_system}. If $j\in\{k+1,\ldots, k+d\}$, let $\bar{u}_j $ denote the image of~$L^{\sigma_\Theta}_{e_j}$ under the quotient map. Thus $\{\bar{u}_{j}\mid  j=k+1,\ldots,k+d\}$ is a set of commuting unitaries.  Hence there is a unique unitary representation $\bar{u}\colon \ZZ^d\cong 0_k\times\ZZ^d\to \Tt_r(\NN^n, \sigma_{\Theta})/\Hilm[I]_d$ mapping a canonical generator $e_{j}\in 0_k\times\ZZ^d$ to $\bar{u}_j$, $j=k+1, k+2, \ldots k+d$. In addition, we have for $p\in\NN^k\times 0_d$ and $j = k+1, k+2, \ldots, k+d$, $$\bar{u}_j\bar{v}_p=e^{-\pi i \braket{e_j}{ \Theta p}}\bar{v}_{e_j+p}= e^{-2\pi i \braket{e_j}{ \Theta p}}\bar{v}_p\bar{u}_j= e^{2\pi i \braket{p}{\Lambda e_j}}\bar{v}_p\bar{u}_j.$$ This guarantees that the isometries $\{\bar{v}_p\mid p\in\NN^k\times 0_d\}$ and the unitaries $\{\bar{u}_x\mid x\in 0_k\times \ZZ^d\}$ satisfy the commuting relations required in~\cite[Section 2]{AaHRS}. We then get a  homomorphism from $B_{\Lambda}$ onto $\Tt_r(\NN^n, \sigma_{\Theta})/\Hilm[I]_d$ induced by the maps $\NN^k\times 0_d\ni p\mapsto \bar{v}_p$, $0_k\times\ZZ^d\ni x \mapsto \bar{u}_x$. For the inverse map, we use the description of $\Tt_r(\NN^n, \sigma_{\Theta})$ as a universal $\Cst$\nb-algebra generated by~$n$ isometries satisfying the relations ($\mathcal{R}_\Theta$) given in Proposition~\ref{pro:presentationtoeplitztori}. Notice that the second relation in $(\mathcal{R}_\Theta)$ is satisfied because $u_{e_j}v_{e_i}=e^{-2\pi i \braket{e_j}{\Theta e_i}}v_{e_i}u_{e_j}$ in $B_\Lambda$ implies $$u_{e_j}^*v_{e_i}=e^{2\pi i \braket{e_j}{\Theta e_i}}v_{e_i}u_{e_j}^*.$$ Thus the map that sends $L^{\sigma_\Theta}_{e_i}$ to $v_{e_i}$ for $i=1,\ldots, k$ and $L^{\sigma_\Theta}_{e_j}$ to $u_{e_{j-k}}$ for $j=k+1,\ldots, k+d$ gives a  homomorphism from $\Tt_r(\NN^n, \sigma_{\Theta})$ onto $B_{\Lambda}$.  The resulting  homomorphism $\Tt_r(\NN^n, \sigma_{\Theta})\to B_\Lambda$ factors through the quotient $\Tt_r(\NN^n, \sigma_{\Theta})/\Hilm[I]_d$, giving an inverse map for the canonical map $B_{\Lambda}\to \Tt_r(\NN^n, \sigma_{\Theta})/\Hilm[I]_d$ built above. This finishes the proof.
\end{proof}

 In contrast to \cite{AaHRS}, our approach to Toeplitz noncommutative tori does not require either of the subsets $\{L^{\sigma_\Theta}_{p} \mid p\in\NN^k\times 0_d\}$ and  $ \{L^{\sigma_\Theta}_{x}\mid x\in 0_k\times \NN^d\}$ to consist  of mutually commuting isometries. So our antisymmetric matrix $\Theta$ can have nontrivial diagonal blocks. 

We also do not require any of the isometries to be unitaries, as is the case for the generators $u_{e_1}, \ldots ,u_{e_d}$  considered in \cite{AaHRS}. Nevertheless, for the dynamics considered in the following sections, KMS states necessarily factor through the quotient of $\Tt_r(\NN^n,\sigma_\Theta)$ 
in which the generating isometries that are fixed by the dynamics become unitaries, see \lemref{GNSunitaries} below. This confirms and explains the insight behind the choice made in \cite[Section 2]{AaHRS} where a subset of  generators are chosen to be unitaries from the onset, and the dynamics acts nontrivially only on the others.

\section{Dynamics and  KMS states}\label{sec:characterisation}
We fix an $n\times n$ antisymmetric matrix~$\Theta=(\theta_{ij})$ throughout this section and we let $\Tt_r(\NN^n,\sigma_{\Theta})$ be the associated $n$\nb-dimensional Toeplitz noncommutative torus. When the cocycle in question is understood and there is no chance of confusion, we will omit it from the notation of the left regular $\sigma_{\Theta}$\nb-representation of~$\NN^n$, simply writing $L_p$ instead of $L_p^{\sigma_\Theta}$ for a canonical generating isometry. So $\{L_p\mid p\in\NN^n\}$ is the canonical set of isometries on $\ell^2(\NN^n)$ satisfying \begin{equation*}
L_pL_q=\sigma_\Theta(p,q)L_{p+q}=e^{-\pi i \braket{p}{\Theta q}}L_{p+q}, \quad\text{and}\quad L_pL_p^*L_qL_q^*=L_{p\vee q}L_{p\vee q}^*,\end{equation*} and thus also condition (3) of \defref{dfn-covariance}.

Let $r\in\RR^n$ and let 
 $\braket{p}{r}=\sum_{\substack{j=1}}^np_jr_j$ be the usual scalar product of $p$ and $r$. For each $t\in \RR$ the map  
$$u_t^r\colon \delta_q \mapsto  e^{i\braket{q}{r}t}\delta_q \qquad (q\in \NN^n)$$ 
extends to a (diagonal)  unitary operator $u^r_t$  on $\ell^2(\NN^n)$ and   $\{u^r_t\mid t\in \RR\}$ 
is a strongly continuous one-parameter unitary group. The automorphism group of $\B(\ell^2(\NN^n))$ 
obtained by conjugation with  $\{u^r_t\mid t\in \RR\}$  satisfies 
$$ u^r_t L_p (u^r_t)^*=  e^{i\braket{p}{r}t}L_p\qquad (p\in \NN^n),$$
and hence restricts to a strongly continuous one-parameter  automorphism group $\{\alpha^r_t\mid t\in \RR\}$  of $\Tt_r(\NN^n,\sigma_{\Theta})$. 
This automorphism group is characterized by its action on generators:
$$\alpha_t^r(L_p)= e^{i\braket{p}{r}t}L_p\qquad (p\in \NN^n).$$
An alternative way to obtain the dynamics is to observe that the vector $r$ determines a group homomorphism $t\mapsto e^{i( r)t} = (e^{i r_1t}, e^{i r_2t}, \ldots ,e^{i r_n t})$  from  $\RR$ to $\TT^n$,
and hence $\alpha^r_t =\gamma_{e^{i( r)t}} $ is simply the composition of this homomorphism with the canonical  gauge action $\gamma$ of~$\TT^n$ on~$\Tt_r(\NN^n,\sigma)$.

 We wish to study the $\mathrm{KMS}$ states for the dynamics associated to the vector $r$. Let $\beta\in\RR$ and recall that a state $\varphi$ of $\Tt_r(\NN^n,\sigma_{\Theta})$ is a \kmsb state for the automorphism group $\{\alpha^r_t\mid t\in \RR\}$ of $\Tt_r(\NN^n,\sigma_{\Theta})$ if it satisfies the \kmsb condition \begin{equation}\label{kmscondition}
\varphi(AB)=\varphi(B\alpha^r_{i\beta}(A))
\end{equation}
 for~$A$ $\alpha^r$\nb-analytic and $B \in \Tt_r(\NN^n,\sigma_{\Theta})$. It is well known that it suffices to verify this for~$A$ and~$B$ in an $\alpha^r$\nb-invariant set of analytic elements with dense linear span, such as the 
generating elements $L_pL_q^*$ with $p,q \in \NN^n$, see \cite[Proposition 8.12.3]{ped}.

It is immediate that $e^{-\braket{p}{r}\beta}=\varphi(L_pL_p^*) \leq 1$ for every  \kmsb state $\varphi$  and each $p\in \NN^n$. Hence, if $r_j<0$ for some $j\in\{1,2,\ldots,n\}$, then there is no \kmsb state for~$\beta>0$, and similarly, if $r_j>0$ for some $j\in\{1,2,\ldots,n\}$, then there is no \kmsb state for~$\beta<0$. It follows that for a \kmsb state to exist at nonzero $\beta$, each~$r_j$ must have the same sign as~$\beta$. Here we choose to work with nonnegative inverse temperatures. From now on we fix a vector~$r\in[0,\infty)^n$, which we often omit  from the notation, writing~$\alpha$ instead of~$\alpha^r$ when there is no risk of confusion.

We begin by verifying that the generating  isometries corresponding to the vanishing coordinates of $r$ become unitary operators in the GNS representation of any \kmsb state.

\begin{lemma} \label{GNSunitaries} Let~$\varphi_\beta$ be a \kmsb state of~$(\Tt_r(\NN^n,\sigma_{\Theta}),\alpha)$. Let~$(\Hilm[H]_{\varphi_\beta},\pi_{\varphi_\beta})$ denote the associated GNS representation. If $p\in\NN^n$ and $\braket{p}{r}=0$, then $\pi_{\varphi_\beta}(L_p)$ is unitary. Moreover, $\varphi_\beta$ factors through the quotient of $\Tt_r(\NN^n,\sigma_\Theta)$ modulo the ideal generated by the projections $\{1-L_{e_j}L_{e_j}^*\mid r_j=0\}$.

\begin{proof} Set $Q_p\coloneqq 1-L_pL_p^*$. Then $Q_p$ is analytic, and by the KMS condition,
\[\varphi_{\beta}(Q_p)=  1-\varphi_{\beta}(L_pL_p^*) = 1-e^{-\braket{p}{r}\beta}= 1-1= 0.\]
 We know that the elements of the form $L_pL_q^*$ have dense linear span in~$\Tt_r(\NN^n,\sigma_{\Theta})$ and satisfy $\alpha_z(L_pL_q^*) = e^{i\braket{p-q}{r}z} L_pL_q^*$. Hence the hypotheses of \cite[Lemma 2.2]{aHLRS13} are satisfied, and so we conclude that $\varphi$ factors through the quotient of $\Tt_r(\NN^n,\sigma_{\Theta})$ by the ideal generated by the projections $\{Q_p\mid \braket{p}{r}=0\}$.  This implies that~$\varphi_{\beta}(B^*Q_pB)=0$ for every~$B\in\Tt_r(\NN^n,\sigma_{\Theta})$, which implies  that 
$\pi_{\varphi_\beta}(L_p)$ is unitary as wished. In order to complete the proof of the lemma, notice that the ideal generated by the projections $\{Q_p\mid \braket{p}{r}=0\}$ coincides with the one generated by  $\{1-L_{e_j}L_{e_j}^*\mid r_j=0\}$.
\end{proof}
\end{lemma}

Given $d\in\NN$, we denote by $0_d$ the zero element of $\NN^d$, so that 
$$0_d=\underbrace{(0,\ldots,0)}_{d \text{ times}}.$$ 
We shall assume without loss of generality that all the nonzero coordinates of~$r$ appear at the beginning. Thus we fix the notation 
 $r = (r_1, r_2, \ldots r_k, 0_d)$ to indicate a vector in $[0,\infty)^n$ with strictly positive first $k$ coordinates and zeros for the remaining $d$ coordinates. This includes the two extreme cases of  $r=0$, in which case  we have $k =0$,  and of strictly positive $r$, in which case~$k =n$.

\begin{proposition}\label{condexp} Let $n = k+d$ with $k,d\in\NN$ and let $E^{(k)}\coloneqq E^{\TT^k \times \{1_d\} }$ denote the conditional expectation associated to the restriction of the gauge action of $\TT^n$ to the closed subgroup~$\TT^k \times \{1_d\}$. Then 

\begin{enumerate}

\item $\Tt_r(\NN^n,\sigma_\Theta) = \clsp\{L_p L_x L_y^* L_q^*\mid p,q \in \NN^k\times 0_d, \ x,y \in 0_k\times \NN^d\};$

\smallskip
\item $
E^{(k)} (\Tt_r(\NN^n,\sigma_\Theta)) =  \clsp \{ L_p L_x L_y^* L_p^*\mid p \in \NN^k\times 0_d,\ x,y \in 0_k\times \NN^d\}.
$
\end{enumerate}
\begin{proof}
That the elements of the form $L_p L_x L_y^* L_q^*$ with $p,q \in \NN^k\times 0_d$ and $x,y \in 0_k\times \NN^d $ span a dense \Star subalgebra of $\Tt_r(\NN^n,\sigma_{\Theta})$ is a consequence of relations (1)--(3) of Definition~\ref{dfn-covariance}. The second assertion follows from part (1)  because $E^{(k)}$ is a contraction that satisfies  $E^{(k)} ( L_p L_x L_y^* L_q^*) = \delta_{p,q} L_p L_x L_y^* L_p^*$, where $\delta_{p,q}$ is the Kronecker delta.
\end{proof}
\end{proposition}

\begin{proposition}\label{prop:KMS-characterisation} 
Let $n = k+d$ with $k,d\in\NN$ and let  $\alpha$ be the dynamics determined by $r=(r_1,\ldots, r_k, 0_d)$. Let $0<\beta<\infty$ and suppose that $\varphi$ is a \kmsb state of $(\Tt_r(\NN^n,\sigma_{\Theta}), \alpha)$.
 Then $\varphi $ restricts to a trace on the $\Cst$\nb-subalgebra $\Cst(L_x: x\in  0_k\times\NN^d)$ and satisfies
\begin{align}\label{equ:charac-KMS}
\varphi(L_pL_xL_y^*L_q^*)=\delta_{p,q}e^{-\beta\braket{p}{r}} \varphi(L_xL_y^*)
\end{align}
for all~$p,q\in\NN^k\times 0_d$ and $x,y \in 0_k \times \NN^d$, where $\delta_{p,q}$ is the Kronecker delta. 

\bp 
Suppose that $\varphi$ is a $\mathrm{KMS}_\beta$ state of $(\Tt_r(\NN^n,\sigma_\Theta),\alpha)$. Since the elements in $\Cst(L_x: x\in  0_k\times\NN^d)$ are fixed points of $\alpha$, the \kmsb condition implies that the restriction of $\varphi$ to  $\Cst(L_x: x\in  0_k\times\NN^d)$ is a trace.

We show next that $\varphi$ satisfies  \eqref{equ:charac-KMS}. Let $p,q\in \NN^k\times 0_d$ and $x,y\in 0_k\times\NN^d$. In case $p=q$, the \kmsb condition with $A =L_p$ and $B=L_xL_y^*L_q^*$ gives \eqref{equ:charac-KMS} because  $L_p^*L_q=1$. So suppose $p\neq q$. We aim to prove that $$\varphi(L_pL_xL_y^*L_q^*)=0.$$ For this, we write $p'\coloneqq(p\vee q)-p$ and $q'\coloneqq(p\vee q)-q$. Notice that at least one of the numbers $p', q'$ is nonzero. Because $\varphi$ is a state and so preserves adjoints, we may assume without loss of generality that $q'\neq 0$. Applying the KMS condition \eqref{kmscondition} with $A = L_pL_xL_y^*$ and $B=L_q^*$ and using that $\sigma_\Theta$ is a circle-valued function we get
 \[\big|\varphi(L_pL_xL_y^*L_q^*)\big|=e^{-\beta\braket{p}{r}}\big|\varphi(L_q^*L_pL_xL_y^*)\big|=e^{-\beta\braket{p}{r}}\big|\varphi(L_{q'}L_{p'}^*L_xL_y^*)\big|.\]
 Since $p'\vee x=x+p'$, we have $L_{p'}^*L_x=\sigma_{\Theta}(x,p')^2L_xL_{p'}^*$. Thus
  \begin{equation}\label{first-pq}
  \big|\varphi(L_pL_xL_y^*L_q^*)\big|=e^{-\beta\braket{p}{r}}\big|\varphi(L_{q'}L_xL_{p'}^*L_y^*)\big|=e^{-\beta\braket{p}{r}}\big|\varphi(L_{q'}L_xL_y^*L_{p'}^*)\big|.
 \end{equation}
  Observe that $p'\vee q'=p'+q'$ and $p'\vee y=p'+y$. Hence the same argument as above gives
  \begin{equation*}
  \big|\varphi(L_{q'}L_xL_y^*L_{p'}^*)\big|=e^{-\beta\braket{q'}{r}}\big|\varphi(L_{q'}L_xL_y^*L_{p'}^*)\big|.
  \end{equation*}  
Continuing this process, we see that for each $l\in \NN$ \begin{equation}\label{second-q'p'}
\big|\varphi(L_{q'}L_xL_y^*L_{p'}^*)\big|=e^{-l\beta\braket{q'}{r}}\big|\varphi(L_{q'}L_xL_y^*L_{p'}^*)\big|.
\end{equation} Substituting  \eqref{second-q'p'} into \eqref{first-pq} we arrive at
\[
  \big|\varphi(L_pL_xL_y^*L_q^*)\big|=e^{-\beta\braket{p}{r}}e^{-l\beta\braket{q'}{r}}\big|\varphi(L_{q'}L_xL_y^*L_{p'}^*)\big|.
\]
Since $q'\neq 0$, it follows that $\lim_{\substack{l\to\infty}}e^{-l\beta \braket{q'}{r}}=0.$ Therefore $\varphi(L_pL_xL_y^*L_q^*)=0$ as wished. 
\ep
\end{proposition}

\begin{lemma}\label{lemma:Q-properties}
Let $n = k+d$ with $k,d\in\NN$. The product
\begin{equation}\label{eqn:projQ}
Q \coloneqq \prod_{j =1}^k( 1 - L_{e_i} L_{e_i}^*)
\end{equation}
is a projection in $\Tt_r(\NN^n,\sigma_\Theta)$ satisfying
\begin{enumerate}
\item $ QL_{p} = 0 = L_{p}^* Q $ for every $p \in \NN^k\times 0_d \setminus \{0\} $;
\item $Q L_p^*L_xL_y^*L_p=   L_p^*L_xL_y^*L_pQ = Q L_p^*L_xL_y^*L_p Q$ for every $x,y \in 0_k\times \NN^d$ and $p\in \NN^k\times 0_d$;
\item $Q \Tt_r(\NN^n,\sigma_\Theta) Q = \clsp\{ QL_x L_y^* Q \mid x,y \in 0_k \times \NN^d\} $.
\end{enumerate}
\bp
By definition, $Q$ is the product of a finite collection of commuting projections, thus is a projection as well.
In order to prove (1), let
$0\neq p\in \NN^k \times 0_d $ and let~$j$ be the first nonzero coordinate of~$p$. Then the product $QL_p $ contains a factor of the form $(1 - L_{e_j}L_{e_j}^*) L_{e_j}$ and hence vanishes. Taking adjoints shows that $L_p^* Q $ vanishes too.

In order to prove (2), notice first that 
 $e_i\vee x = e_i+x$ for every   $x \in 0_k \times \NN^d$ and every $i =1, 2, \ldots, k$. Thus by \lemref{lem:relprimecomm}
 the projection $L_{e_i}L_{e_i}^*$ commutes with 
 $L_x$, with every $L_y^*$, and with every $L_xL_y^*$ for $y  \in 0_k \times \NN^d$. 
Thus $Q$ commutes with $L_xL_y^*$ for all $x,y\in 0_k\times\NN^d$. Now take $p\in \NN^k\times 0_d$ and notice that 
\defref{dfn-covariance}(3) applied twice gives 
$$L_p^*L_xL_y^*L_p=\sigma_{\Theta}(x,p)^2\sigma_{\Theta}(p,y)^2L_xL_y^*.$$
 So $Q$ commutes with $L_p^*L_xL_y^*L_p$ because it commutes with $L_xL_y^*$. This completes the proof of part (2).

For part (3), we use that $$\Tt_r(\NN^n,\sigma_\Theta) = \overline{\operatorname {span}} \{L_p L_x L_y^* L_q^* \mid p,q \in \NN^k\times 0_d, \ x,y\in 0_k\times \NN^d\}.$$ So the products $Q L_p L_x L_y^* L_q^* Q $ span a dense subset of the corner.
If $p,q \in \NN^k \times 0_d$ are not both zero, and if $x,y \in 0_k \times \NN^d$, then $Q L_p L_x L_y^* L_q^* Q = 0$ because of part (1). 
Hence the corner is the closed linear span of the products $Q L_x L_y^* Q $ with $x,y\in 0_k\times \NN^d$.
\ep
\end{lemma}

\begin{lemma}\label{lemma:reconstruction}
Let $0 < \beta < \infty$ and suppose that $\varphi$ is a \kmsb state for $(\Tt_r(\NN^n,\sigma_\Theta), \alpha)$. Let $Q$ be the projection from \eqref{eqn:projQ}. Then $\varphi(Q) >0$ and the sum $Z(\beta)\coloneqq\sum_{p\in \NN^k\times 0_d} e^{-\beta \braket{p}{r}}$ satisfies $Z(\beta) =\varphi(Q)^{-1} $. Moreover, the map $\omega_\varphi\colon X\mapsto Z(\beta)\varphi(X)$ for $X\in Q \Tt_r(\NN^n,\sigma_\Theta) Q$
is a tracial state of the corner  $Q \Tt_r(\NN^n,\sigma_\Theta) Q$ and  $\varphi$ can be reconstructed 
from $\omega_\varphi$ by
\begin{equation}\label{reconstruction-formula}
\varphi(X)= \frac{1}{Z(\beta)} \sum_{p\in \NN^k\times 0_d} e^{-\beta \braket{p}{r}} \omega_\varphi( QL_p^* X L_p Q)\quad \text{for  all } X\in\Tt_r(\NN^n,\sigma_\Theta).
\end{equation}
\end{lemma}
\bp
Expanding the product defining~$Q$, we have 
\begin{equation}\label{equ:Q-sum}
Q=\sum_{J\subseteq \{1,\dots, k\}}(-1)^{|J|}\prod_{j \in J}L_{e_j}L_{e_j}^*
\end{equation} since $\{L_pL_p^*\mid p\in\NN^n\}$ are commuting projections.
Let  $i,j\in\{1,\ldots,k\}$ with $i\neq j$, so that $L_{e_i}L_{e_i}^*$ commutes with $L_{e_j}$ by \lemref{lem:relprimecomm}. Then 
\[
\varphi(L_{e_i}L_{e_i}^*L_{e_j}L_{e_j}^*) =\varphi(L_{e_j}L_{e_i}L_{e_i}^*L_{e_j}^*)= e^{-\beta( r_j+r_i)}
\]
by the KMS condition. An analogous computation shows that for each $J\subset \{1,\dots, k\}$, we have
\[\varphi\big(\prod_{j \in J}L_{e_j}L_{e_j}^*\big)=e^{-\beta \sum_{j\in J}r_j}.\]
It follows that 
\begin{align}\label{eqn:phiofQcomputation}
\varphi(Q)= \sum_{J\subseteq \{1,\dots, k\}}(-1)^{|J|}\varphi\big(\prod_{j \in J}L_{e_j}L_{e_j}^*\big) 
=\sum_{J\subseteq \{1,\dots, k\}}(-1)^{|J|}e^{-\beta \sum_{j\in J}r_j} =\prod_{j=1}^k(1-e^{-\beta r_j}),
\end{align}
which is obviously positive and  equal to $Z(\beta)^{-1}$ by an Euler product like expansion.

If $p\neq l$, then
 \[L_pQL_p^*L_{l}QL_{l}^*=\overline{\sigma_\Theta(p,p\vee l-p)}\sigma_\Theta(l,p\vee l-l)L_pQL_{p\vee l-p}L_{p\vee l-l}^*QL_{l}^*=0\]
 by Lemma~\ref{lemma:Q-properties}(1) because $p\vee l-p$ and $p\vee l-l$ cannot both be zero. 
 Hence  the  sum of mutually orthogonal projections
$$P_q\coloneqq\sum_{ \substack{p\leq q}}L_pQL_p^*$$  is a projection itself and satisfies 
 \begin{align*}
\varphi( P_q)
=\sum_{\substack{p\leq q}}\varphi(L_pQL_p^* )
=\sum_{\substack{p\leq q}}e^{-\beta\braket{p}{r}}\varphi(L_p^*L_pQ)
=\sum_{\substack{p\leq q}}e^{-\beta \braket{p}{r}}\varphi(Q).
\end{align*}
It follows that $\varphi(P_q )\to1$ as $q\to \infty$ in the  directed set $\NN^k \times 0_d$ since $\sum_{\substack{p\leq q}}e^{-\beta \braket{p}{r}}\to Z(\beta)=\varphi(Q)^{-1}$ by the first part of the proof. Hence an application of the Cauchy--Schwarz inequality implies that $\varphi(P_q X P_q)\to \varphi(X)$ as $q\to \infty$ for all $X\in \Tt_r(\NN^n,\sigma_\Theta)$. Thus by the KMS condition we have
\begin{align*}
\varphi(X)&=\sum_{p,l\in \NN^k\times 0_d}\varphi(L_pQL_p^* X L_lQL_l^*)\\
&=\sum_{p\in \NN^k\times 0_d}\varphi(L_pQL_p^* X L_pQL_p^*)\\
&=\sum_{p\in \NN^k\times 0_d}e^{-\beta \braket{p}{r}}\varphi(QL_p^* X L_pQ)
\end{align*}
We then get the reconstruction formula \eqref{reconstruction-formula} by replacing the summand $\varphi(QL_p^* X L_pQ)$ by $\frac{1}{Z(\beta)}\omega_\varphi(QL_p^* X L_pQ)$ .
\ep

\begin{proposition}\label{isomorphisms kmsbeta states}
Let $0 < \beta < \infty$ and let $Q $ be the projection from \eqref{eqn:projQ}. For each tracial state~$\omega$ of the corner  $Q \Tt_r(\NN^n,\sigma_\Theta) Q$,  define
\begin{equation}\label{eqn:Tbeta0}
T_{\beta}( \omega)(X) \coloneqq \frac{1}{Z(\beta)} \sum_{l\in \NN^k\times 0_d} e^{-\beta\braket{l}{r}} 
\omega(Q L_l^* X L_l Q), \qquad X\in \Tt_r(\NN^n,\sigma_\Theta).
\end{equation}
Then $T_\beta$ is an affine weak* 
homeomorphism of the tracial state space of the corner onto the \kmsb state space of $(\Tt_r(\NN^n,\sigma_\Theta), \alpha)$.
\end{proposition}
\bp
Let $\omega$ be a tracial state on the corner $Q \Tt_r(\NN^n,\sigma_\Theta) Q$. Clearly $T_{\beta}( \omega)$ is a state of $\Tt_r(\NN^n,\sigma_\Theta)$. 
First we verify that $T_{\beta}( \omega)$ satisfies \eqref{equ:charac-KMS}. Let $X=L_p L_xL_y^*L_q^*$, where $p,q\in\NN^k\times 0_d$, $x,y\in 0_k\times \NN^d$. Given $ l\in \NN^k\times 0_d$,  Lemma~\ref{lemma:Q-properties}(1) implies that if the elements $l\vee p-l $ and $l\vee q-l$ are not both zero, then $Q L_l^* X L_l Q=0$. Since $l\vee p-l=0=l\vee q-l $ is equivalent to $l\geq p\vee q$, we may restrict the sum to $l\geq p\vee q$ and simplify $L_l^*L_p = \sigma_\Theta(p,l-p) L^*_{l-p}$ using \defref{dfn-covariance} to get
\begin{align}\label{eqn:Tbeta}
T_{\beta}( \omega)(X)&=\frac{1}{Z(\beta)} \sum_{\substack{\underset{l\geq p\vee q}{l\in \NN^k\times 0_d}}}\tfrac{\sigma(p,l-p)}{\sigma(q,l-q)}  \, e^{-\beta\braket{l}{r}} \, \omega(Q L_{l-p}^* L_xL_y^*L_{l-q} Q).
\end{align}
When $p=q$, a change of the index of summation to $m = l-p$ gives \begin{eqnarray}\label{eqn:Tbetapradepois}
T_{\beta}( \omega)(X)&=&\frac{1}{Z(\beta)} \sum_{\substack{\underset{l\geq p}{l\in \NN^k\times 0_d}}} \, e^{-\beta\braket{l}{r}}  \omega(Q L_{l-p}^* L_xL_y^*L_{l-p} Q) \notag\\
&=&\frac{1}{Z(\beta)} \sum_{\substack{\underset{m\geq 0}{m\in \NN^k\times 0_d}}} \, e^{-\beta \braket{p}{r}} e^{-\beta \braket{m}{r}}  \omega(Q L_{m}^* L_xL_y^*L_{m} Q) \\
&= &e^{-\beta \braket{p}{r}}T_{\beta}( \omega)(L_xL_y^*), \notag
\end{eqnarray} To see that $T_\beta(\omega)(X)=0$ if $p\neq q$, it suffices to show $Q L_{l-p}^* L_xL_y^*L_{l-q} Q$ vanishes for all $l\geq p \vee q$, $l\in\NN^k\times 0_d$. By \lemref {lem:relprimecomm}  the projection $L_{l-p}L_{l-p}^*$ commutes $L_sL_t^*$ with $s,t\in 0_k\times\NN^d$. In addition, it follows from Lemma~\ref{lemma:Q-properties}(2) that $L_{l-p}^* L_xL_y^*L_{l-p}$ commutes with~$Q$. So
\begin{equation*}\label{equ:vp-l}
\begin{aligned}
QL_{l-p}^* L_xL_y^*L_{l-q}Q&=QL_{l-p}^* L_xL_y^*L_{l-p}L_{l-p}^*L_{l-q} Q=QL_{l-p}^* L_xL_y^*L_{l-p}QL_{l-p}^*L_{l-q} Q.\end{aligned}
\end{equation*} Since $p\neq q$ implies that at least one of the elements  $((l-p)\vee (l-q))-(l-p)$ and $((l-p)\vee (l-q))-(l-q)$ is nonzero, \lemref{lemma:Q-properties}(1) yields $QL_{l-p}^*L_{l-q} Q=0$ and hence $QL_{l-p}^* L_xL_y^*L_{l-q}Q=0$. Therefore $T_{\beta}( \omega)(X)=0$ when $p\neq q$, as asserted. 

We aim to prove next that $T_\beta(\omega)$ satisfies the \kmsb condition for the dynamics $\alpha$. Let $p,q,a,b \in \NN^k \times 0_d$ and $x, y, s,t \in 0_k \times \NN^d$ and consider the elements $X= L_pL_xL_y^*L_q^*$ and $Y = L_aL_sL_t^*L_b^*$, so that $\alpha_{i\beta}(X) =  e^{-\beta \braket{p-q}{r}} X $. In order to conclude that $T_\beta(\omega)$ is a \kmsb state it suffices to show that
$ T_{\beta}( \omega)(XY)  e^{\beta \braket{p}{r}} =  e^{\beta \braket{q}{r}}T_{\beta}( \omega)( YX)$. We may assume that $p-q+ a -b =0$, for otherwise both sides vanish because we have shown that $T_\beta(\omega)$ satisfies \eqref{equ:charac-KMS}. 
By definition 
\[ 
T_\beta(XY) =    \sum_{\substack{l\in \NN^k\times 0_d}} \frac{e^{-\beta\braket{l}{r}}}{Z(\beta)}\omega(QL_l^*L_pL_xL_y^*L_q^* L_aL_sL_t^*L_b^* L_l Q). 
\]
The argument leading to \eqref{eqn:Tbeta} shows that we may restrict the summation to  $l\geq p\vee b$, in which case
$L_l^* L_p = \sigma_\Theta(p,l-p)L_{l-p}^*$ and similarly $L_b^* L_l = \overline{\sigma_\Theta(b,l-b)} L_{l-b}$ by \defref{dfn-covariance}(3). Thus 
 \begin{equation}\label{vlequ1}
\begin{aligned}
T_\beta(XY) &= \sum_{\substack{\underset{l\geq  p\vee b}{l\in \NN^k\times 0_d}}} \tfrac{e^{-\beta\braket{l}{r}}}{Z(\beta)}\tfrac{\sigma(p,l-p)}{\sigma(b,l-b)}
\, \omega(Q L_{l-p}^* L_xL_y^*L_q^* L_aL_sL_t^*L_{l-b} Q) \\
&=  \sum_{\substack{\underset{l\geq  p\vee b}{l\in \NN^k\times 0_d}}} \tfrac{e^{-\beta\braket{l}{r}}}{Z(\beta)}\tfrac{\sigma_\Theta(p,l-p)}{\sigma_\Theta(b,l-b)}
\, \omega(Q L_{l-p}^* L_xL_y^* (L_{l-p}L_{l-p}^*) L_q^* L_a (L_{l-b}L_{l-b}^*)L_sL_t^*L_{l-b} Q)\\
&= \sum_{\substack{\underset{l\geq  p\vee b}{l\in \NN^k\times 0_d}}} \tfrac{e^{-\beta\braket{l}{r}}}{Z(\beta)}\tfrac{\sigma_\Theta(p,l-p)}{\sigma_\Theta(b,l-b)}
\, \omega((Q L_{l-p}^* L_xL_y^* L_{l-p}Q) L_{l-p}^* L_q^* L_a L_{l-b}  (Q L_{l-b}^*L_sL_t^*L_{l-b} Q)), 
\end{aligned}
\end{equation}  
where we have used that $L_{l-p}L_{l-p}^*$ commutes with $L_xL_y^*$, that  $L_{l-b}L_{l-b}^*$ commutes with 
$L_sL_t^*$, and that $Q$ commutes with $L_{l-p}^* L_xL_y^* L_{l-p} $ and with $L_{l-b}^*L_sL_t^*L_{l-b}$. Notice also that the product in the middle simplifies to a scalar, namely
\[
L_{l-p}^* L_q^* L_a L_{l-b} = \tfrac{\sigma_\Theta(a,l-b)}{\sigma_\Theta(q,l-p)} L_{q+l-p}^* L_{a+l-b} = \tfrac{\sigma_\Theta(a,l-b)}{\sigma_\Theta(q,l-p)}
\]
because $q-p = a -b$. When we substitute this in the formula and 
change the index of summation to $m = l-p$ we get \begin{multline}
\label{vlequ2}
e^{\beta \braket{p}{r}} T_\beta(XY) = \\
= \sum_{\substack{\underset{m+p\geq b}{m\in \NN^k\times 0_d}}} \tfrac{e^{-\beta \braket{m}{r}}}{Z(\beta)}
\tfrac{\sigma_\Theta(p,m)\sigma_\Theta(a, p+m-b)}{\sigma_\Theta(b,m+p-b) \sigma_\Theta(q,m)} \,
\omega((Q L_{m}^* L_xL_y^*L_{m} Q) (Q L_{m+p-b}^*L_sL_t^*L_{m+p-b} Q)).
\end{multline}

Exchanging now the roles of $X$ and $Y$ and carrying out a computation like \eqref{vlequ1} gives
 \begin{equation*}\label{vlequ3}
\begin{aligned}
T_\beta(YX) &=\\
&=\sum_{\substack{\underset{l\geq a\vee q}{l\in \NN^k\times 0_d}}} 
\tfrac{e^{-\beta\braket{l}{r}}}{Z(\beta)}\tfrac{\sigma_\Theta(a,l-a)}{\sigma_\Theta(q,l-q)}
\,\omega((Q L_{l-a}^* L_sL_t^* L_{l-a}Q) L_{l-a}^* L_b^* L_p L_{l-q}  (Q L_{l-q}^*L_xL_y^*L_{l-q} Q))
\end{aligned}
\end{equation*}  
Substituting the scalar $L_{l-a}^* L_b^* L_p L_{l-q} =\overline{\sigma_\Theta(b,l-a)}  \sigma_\Theta(p, l-q)$ in the middle  and changing the index of summation to $m = l-q$,  we get
\begin{multline}
\label{vlequ4}
e^{\beta \braket{q}{r}}T_\beta(YX) = \\
=  \sum_{\substack{\underset{m+q\geq  a}{m\in \NN^k\times 0_d}}} 
 \tfrac{e^{-\beta \braket{m}{r}}}{Z(\beta)}\tfrac{\sigma_\Theta(a,m+q-a) \sigma_\Theta(p, m)}{\sigma_\Theta(q,m) \sigma_\Theta( b,q+m-a)}
\, \omega((Q L_{m+q-a}^* L_sL_t^* L_{m+q-a}Q) (Q L_{m}^*L_xL_y^*L_{m} Q)).
\end{multline}

Since $q-a =p-b$,  so that $m+q\geq a$ iff $m + p \geq b$, and since
$\omega$ is a trace on the corner, the two series in \eqref{vlequ2} and \eqref{vlequ4} are the same, term by term, which shows
$ e^{\beta \braket{p}{r}} T_{\beta}( \omega)(XY)   =  e^{\beta \braket{q}{r}}T_{\beta}( \omega)( YX)$. This completes the proof that $T_\beta$ maps tracial states of $Q\Tt_r(\NN^n,\sigma_\Theta)Q$ to \kmsb states of $( \Tt_r(\NN^n,\sigma_\Theta),\alpha)$.

\bigskip

 Notice that $T_\beta$ is surjective by Lemma~\ref{lemma:reconstruction}  and  injective because the restriction of $T_\beta(\omega)$  to $Q\Tt_r(\NN^n,\sigma_\Theta)Q$
  is equal to $Z(\beta) \omega$. Clearly it is also an affine map; its inverse is given by the map $\varphi \mapsto \omega_\varphi$ of Lemma~\ref{lemma:reconstruction}, which is obviously weak* continuous. 
 We then conclude that $T_\beta$ is a weak* homeomorphism, as it is a bijection with continuous inverse between compact Hausdorff spaces. This completes the proof of the proposition. 
\ep

\section{\kmsb states and traces on noncommutative tori}\label{sec:traces}
Our first goal in this section is to show that the corner $Q\Tt_r(\NN^n,\sigma_{\Theta})Q$ is isomorphic to  the $\Cst$\nb-subalgebra of~$\Tt_r(\NN^n,\sigma_{\Theta})$ generated by $\{L_x\mid x\in 0_k\times\NN^d\}$, which is itself isomorphic to the Toeplitz noncommutative torus $\Tt_r(\NN^d,\sigma_{\Thetad})$ associated to the restriction $\Thetad$ of~$\Theta$ to the last $d$ coordinates. 
We  then show that the traces of $\Tt_r(\NN^d,\sigma_{\Thetad})$ factorize through its canonical quotient~$\athd$,  and in fact come from states of its center, which is  a classical torus of dimension equal to the degeneracy index of $\Thetad$. 
 
\begin{lemma}\label{lem:iso-corner} Let $Q$ be the projection from \eqref{eqn:projQ} and denote by $\Thetad$ the  lower right $d\times d$ corner of~$\Theta$. Then $\Cst(L_x: x\in  0_k\times\NN^d)$  is canonically isomorphic to the Toeplitz noncommutative torus $ \Tt_r(\NN^d,\sigma_{\Thetad}) $, and the map  $\rho_Q\colon \Cst(L_x: x\in  0_k\times\NN^d) \to Q \Tt_r(\NN^n,\sigma_\Theta) Q$ given by the compression $X\mapsto QXQ$ is an isomorphism.
\bp 
The set of isometries $\{L_x\mid x\in 0_k\times \NN^d\}$ satisfies relations (1)--(3) from Definition~\ref{dfn-covariance}. Using the obvious identification $\NN^d\cong0_k\times \NN^d$, we obtain a covariant isometric $\sigma_{\Thetad}$\nb-representation of~$\NN^d$ in $\Cst(L_x: x\in 0_k\times\NN^d)$. By \corref{cor:faithful_charac}, this gives a  homomorphism from $\Tt_r(\NN^d,\sigma_{\Thetad})$ onto $\Cst(L_x: x\in 0_k\times\NN^d)$ mapping $L_x^{\sigma_{\Thetad}}$ to $L_x$.  It is faithful because $\prod_{\substack{j=k+1}}^n(1-L_{e_j}L_{e_j}^*)\neq 0.$
Notice that $Q$ is the projection of 
$\ell^2(\NN^n)$ onto the subspace $\ell^2(0_k\times \NN^d)$, which is invariant for $\Cst(L_x: x\in  0_k\times\NN^d)$
and hence $\rho_Q\colon \Cst(L_x: x\in  0_k\times\NN^d) \to Q \Tt_r(\NN^n,\sigma_\Theta) Q$ is an isomorphism. 
\ep
\end{lemma}
The isomorphisms from the lemma above allow us to express our characterisation of \kmsb states in terms of traces on $\Tt_r(\NN^d,\sigma_{\Thetad})$.

\begin{proposition} \label{pro:KMSfromtraces} Let $\rho_Q$ be the isomorphism from \lemref{lem:iso-corner} and identify $\Cst(L_x\mid x\in 0_k\times \NN^d) $ with $\Tt_r(\NN^d,\sigma_{\Thetad})$ canonically.
For each tracial state $\tau$ of $\Tt_r(\NN^d,\sigma_{\Thetad})$ there is  a \kmsb state  of $(\Tt_r(\NN^n,\sigma_\Theta),\alpha)$
determined by
\begin{equation}\label{eqn:TbetaEuler}
T_{\beta}( \tau\circ \rho_Q\inv)(L_p L_x L_y^* L_q^*) = \delta_{p,q}\, \tau( L_xL_y^*)  \prod_{j=1}^k\frac{e^{-\beta r_j p_j}(1 - e^{-\beta r_j })}{1 - e^{-\beta r_j +2\pi i \braket{\Theta (x-y)}{e_j}}},
\end{equation} 
where $x,y \in 0_k\times \NN^d \cong \NN^d$. 
The map $\tau \mapsto T_{\beta}( \tau\circ \rho_Q\inv)$ is an affine weak* homeomorphism
of the  tracial state space of $\Tt_r(\NN^d,\sigma_{\Thetad})$ onto the  simplex of \kmsb states of $(\Tt_r(\NN^n,\sigma_\Theta),\alpha)$.
\begin{proof}
The (tracial) states of the corner $Q\Tt_r(\NN^n,\sigma_\Theta)Q$ come from the (tracial) states of 
$\Tt_r(\NN^d,\sigma_{\Thetad})$ via the map $\tau \mapsto \tau \circ \rho_Q\inv $. Combining this with the map $T_\beta$ from \proref{isomorphisms kmsbeta states} we see that $\tau \mapsto T_{\beta}( \tau\circ \rho_Q\inv)$ is an affine weak* homeomorphism
of the  tracial state space of $\Tt_r(\NN^d,\sigma_{\Thetad})$ onto the simplex of \kmsb states of $(\Tt_r(\NN^n,\Theta),\alpha)$. 

In order to  write $T_{\beta}( \tau\circ \rho_Q\inv)$ in \eqref{eqn:Tbeta0} in terms of the tracial state~$\tau$ of 
$\Tt_r(\NN^d,\sigma_{\Thetad})$, we first use  \defref{dfn-covariance}(3) and the fact that $\sigma_\Theta$ is a symplectic bicharacter to write
 $$Q L_{m}^* L_xL_y^*L_{m} Q=\sigma_\Theta(x, m)^2\sigma_\Theta(m,y)^2Q L_x L_y^*Q = \sigma_\Theta(x-y, m)^2 Q L_x L_y^*Q$$ 
 for all $m\in\NN^k\times 0_d$ and $x,y \in 0_k \times \NN^d$. Next, we use  \eqref{eqn:Tbetapradepois}  with $\tau \circ \rho_Q\inv $ playing the role of $\omega$, observing that
 $ \rho_Q\inv (QL_x L_y^*Q) = L_xL_y^*$, and  we conclude that 
 \begin{align}\label{eqn:Tbeta2}
T_{\beta}( \tau\circ \rho_Q\inv)(L_p L_x L_y^* L_q^*) &= \, \tau( L_xL_y^*) \frac{\delta_{p,q}e^{-\beta \braket{p}{r}}}{Z(\beta)} \sum_{m\in \NN^k\times 0_d}\  \, e^{-\beta \braket{m}{r}} \, \sigma_\Theta(x-y,m)^2.
\end{align}
The rest of the proof is a computation to obtain an Euler product formula for the series involving the cocycle. Observe that  $\braket{m}{r} = \braket{r}{m} \geq   0$ and that $\sigma_\Theta(x-y,m)^2 = e^{-2\pi i \braket{x-y}{\Theta m}} =e^{2\pi i \braket{\Theta(x-y)}{m}} $. The series from~\eqref{eqn:Tbeta2} then becomes 
\begin{eqnarray*} 
 \sum_{m\in \NN^k\times 0_d}\  \, e^{-\beta \braket{m}{r}} \, \sigma_\Theta(x-y,m)^2 
&=&  \sum_{m\in \NN^k\times 0_d}\  e^{\braket{-\beta r+2\pi i \Theta(x-y)}{m}}\\
&=&  \sum_{m\in \NN^k\times 0_d}\  e^{\braket{-\beta r+2\pi i \Theta(x-y)}{m}}\\
&=&  \sum_{m\in \NN^k\times 0_d}\  e^{\sum_{j=1}^k\braket{-\beta r+2\pi i \Theta(x-y)}{e_j} \braket{e_j}{m}}\\
&=&  \sum_{m\in \NN^k\times 0_d}\  \prod_{j=1}^k \bigg(e^{\braket{-\beta r+2\pi i \Theta(x-y)}{e_j} }\bigg)^{m_j}.
\end{eqnarray*}
To simplify the notation, we define $A_j \coloneqq e^{\braket{-\beta r +2\pi i \Theta (x-y)}{e_j}}$  and notice that $|A_j| <1$ because $\braket{\beta r}{e_j} >0$ for $1\leq j \leq k$. 
Then 
\[
 \sum_{m\in \NN^k\times 0_d}\  \, e^{-\beta \braket{m}{r}} \, \sigma_\Theta(x-y,m)^2   
= \sum_{m\in \NN^k\times 0_d}\  \prod_{j=1}^k A_j^{m_j}
= \prod_{j=1}^k  \sum_{n=0}^\infty A_j^n 
= \prod_{j=1}^k  \frac{1}{1-A_j} .
\]
Thus we obtain \eqref{eqn:TbetaEuler} by substituting into  \eqref{eqn:Tbeta2} the above expression and the usual Euler product expansion for $Z(\beta)^{-1}$ from \eqref{eqn:phiofQcomputation},  and replacing $e^{-\beta \braket{p}{r}}$ by the product $\prod_{\substack{j=1}}^ke^{-\beta r_jp_j}$. 
 \end{proof}
\end{proposition}

%
%
\def\Tht{D}

The remainder of this section is dedicated to provide a concrete description of the space of tracial states of  $\Tt_r(\NN^d,\sigma_{\Thetad})$. Since our considerations are general, we momentarily adjust the notation and consider a generic~$d\times d$ matrix~$\Tht$. We begin by showing that the traces of $\Tt_r(\NN^d,\sigma_{\Tht})$ factor through the noncommutative torus $\A_{\Tht}$.

\begin{lemma}\label{tracesfromcentre}
Let $\Tht\in M_d(\RR)$ be an antisymmetric matrix and let $\A_\Tht$ be the associated noncommutative $d$\nb-torus, with generating unitaries $U_j$, for $j = 1,\ldots , d$. The map $L_{e_j} \mapsto U_j$  extends to a surjective homomorphism 
$\pi\colon\Tt_r(\NN^d,\sigma_{\Tht}) \to \A_{\Tht}$ which, in turn, induces 
an affine homeomorphism $\tau \mapsto \tau \circ \pi$ from  the space of tracial states of $\A_{\Tht}$ onto the space of tracial states of $\Tt_r(\NN^d,\sigma_{\Tht})$.

\end{lemma}
\begin{proof} By Proposition~\ref{pro:presentationtoeplitztori}, the map that sends an isometry~$L_{e_j}\in\Tt_r(\NN^d,\sigma_{\Tht})$ to the unitary $U_j\in \A_{\Tht}$ for each $1\leq j\leq d$  induces a surjective   homomorphism $\pi\colon\Tt_r(\NN^d,\sigma_{\Tht}) \to \A_{\Tht}$. Hence  the map that sends a tracial state~$\tau$ of $\A_{\Tht}$ to the composite $ \tau\circ\pi$ is an injective affine weak* continuous map from the tracial state space of $\A_{\Tht}$  into the tracial state space  of $\Tt_r(\NN^d,\sigma_{\Tht})$. 
To see that this map is a homeomorphism, it suffices to show that it is also surjective as the underlying spaces are compact and Hausdorff.

Suppose $\psi$ is a tracial state on $\Tt_r(\NN^d,\sigma_{\Tht})$ and let 
$(H_{\psi}, \pi_{\psi})$ be its GNS representation. 
By Lemma~\ref{GNSunitaries} each $\pi_{\psi}(L_x)$ is a unitary operator
and by \proref{pro:presentationtoeplitztori} the collection of unitaries $\{\pi_{\psi}(L_{e_j})\mid j=1,\ldots,d\}$ satisfies the defining relations of $\A_{\Tht} $. So by the universal property of~$\A_{\Tht}$, there is a  homomorphism $\rho\colon\A_{\Tht}\to\pi_\psi(\Tt_r(\NN^d,\sigma_{\Tht}))$ such that  $\rho(U_j) =\pi_{\psi}(L_{e_j})$ for each $j =1, \ldots, d$. Thus
$$\psi(L_{e_j}) = \langle \pi_\psi(L_{e_j}) \xi_\psi, \xi_\psi \rangle = \langle \rho(U_j) \xi_\psi, \xi_\psi\rangle = \langle (\rho\circ \pi)(L_{e_j})  \xi_\psi, \xi_\psi\rangle, $$ showing that~$\psi$ factors through the  homomorphism $\pi\colon\Tt_r(\NN^d,\sigma_{\Tht}) \to \A_{\Tht}$ and finishing the proof of the lemma.
\end{proof}

Following \cite[Section 1]{slawny1972}, see also \cite{C.Phillips}, we say that an antisymmetric, real, $d\times d$ matrix~$\Tht$ is  \emph{nondegenerate} if whenever $x\in \ZZ^d$ and~$\braket{ x}{\Tht y} \in\ZZ$ for all~$y\in\ZZ^d$, then~$x=0$. The noncommutative torus $\A_{\Tht}$ is simple if and only if $\Tht$ is nondegenerate. See \cite[Theorem 3.7]{slawny1972} and also \cite[Theorem 1.9]{C. Phillips}. 

There is always a canonical tracial state on $\A_{\Tht}$, which is given by the conditional expectation $E^\gamma\colon \A_{\Tht}\to \A_{\Tht}^\gamma=\CC,$ where $\gamma$ is the canonical gauge action of $\TT^d$ on~$\A_{\Tht}.$ This is the unique tracial state of $\A_{\Tht}$ when $\Tht$ is nondegenerate (see, for example, \cite[Theorem 1.9]{C. Phillips}). In general, we consider the subgroup 
$$H\coloneqq\{x\in\ZZ^d\mid \braket{x}{\Tht y}\in\ZZ\text{ for all } y\in\ZZ^d\}$$
of $\ZZ^d$. We refer to the rank of $H$ as the {\em degeneracy index}  of $\Tht$. We write $m\coloneqq\operatorname{rank}H$ for the degeneracy index of~$\Tht$ and recall that there exist a basis $\{p_1,p_2,\ldots,p_d\}$ of~$\ZZ^d$ and positive integers $a_1,a_2,\ldots,a_m$ with $a_i|a_{i+1}$ such that 
$\{a_1p_1,\ldots,a_mp_m\}$ is a basis for~$H$ (see, for example, \cite[Theorem~2.6]{fuchs2015abelian}). We consider such a basis in what follows.

\def\gp{\gamma^\Lambda}

\begin{lemma}\label{lem:actionandfpa}
Let $\Tht\in M_d(\RR)$ be an antisymmetric matrix and let $\A_\Tht$
be the associated noncommutative torus, with canonical unitary generators $U_1, \cdots, U_d$.  Let $m=\rank H$  be the degeneracy index of $\Tht$ and let $\{p_1,p_2,\ldots,p_d\}$ be a basis for $\ZZ^d$ as above. Consider the compact group
$$
 \Lambda\coloneqq \ZZ_{a_1}\times \cdots\ZZ_{a_m}\times\TT^{d-m} \subset \TT^d,
 $$
 where $\ZZ_{a_j} $ is the cyclic group of order~$a_j$ viewed as 
 the subgroup of~$\TT$ generated by a primitive $a_j{}^{th}$ root of unity. Then
there is a continuous action $\gamma'$ of 
 $\TT^d$ on $\A_\Tht$ satisfying $\gamma'_\lambda (U_{p_i}) = \lambda_i U_{p_i}$ for all $\lambda = (\lambda_1, \cdots, \lambda_d) \in\TT^d$ and such that if we let $\gp$ denote the restriction of $\gamma'$ to $\Lambda$, $\A_\Tht^{\Lambda}$  denote the corresponding fixed point algebra, and $Z(\A_\Tht)$ denote the center of $\A_\Tht$, then 
 \begin{equation}\label{eqn:fpa=C*U=ZAD}
\Cst(U_b\mid b\in H) =\A_\Tht^{\Lambda} = Z(\A_\Tht).
  \end{equation}

\begin{proof} 
 For each $x = (x_1, x_2, \ldots, x_d) \in \ZZ^{d}$ we set  $U_x\coloneqq U_1^{x_{1}}U_2^{x_{2}}\ldots U_{d}^{x_{d}}.$
 Let $B\in \mathrm{GL}_d(\ZZ)$ be the matrix whose $i$\nb-th column is~$p_i$. 
 Since 
 \[
 U_{p_i}U_{p_j} = e^{-2\pi i\langle{p_i}\mid {\Tht p_j}\rangle}U_{p_j}U_{p_i}=
e^{-2\pi i\langle{B e_i}\mid {\Tht B e_j}\rangle}U_{p_j}U_{p_i} =
e^{-2\pi i(B^T\Tht B)_{i,j}}U_{p_j}U_{p_i}
 \]
  there is a homomorphism  $\pi_B\colon \A_{B^T\Tht B} \to \A_{\Tht}$
  that sends the canonical generator $\widetilde{U}_i\in \A_{B^T\Tht B}$ to $U_{p_i}\in \A_{\Tht}$ for $i=1,\ldots, d$. This is in fact an isomorphism with inverse given by the map $\pi_{B^{-1}}\colon \A_{\Tht}\to \A_{B^T\Tht B}$ that sends the generator $U_i$ to $\widetilde{U}_{B^{-1}e_i}$, $i=1,\ldots,d$ (see \cite[Remark~1.2]{C.Phillips}). The action $\gamma'\colon \TT^d\to\mathrm{Aut}(\A_\Tht)$ as in the statement of the lemma can now be defined by using the isomorphism $\pi_B$ to conjugate  the canonical gauge action of $\TT^d$ on 
 $\A_{B^T\Tht B}$ into an action on $\A_\Tht$. This proves the first assertion.

 In order to prove the first equality in \eqref{eqn:fpa=C*U=ZAD}, observe that
 $b\in H$ iff  $b=x_1a_1p_1+\cdots +x_ma_mp_m $ for some $x_1, \ldots , x_m \in \ZZ$. So the product $U_b = U_1^{b_1}\ldots U_d^{b_d}$ can be rearranged using the cocycle to yield
\[
U_b = \alpha U^{a_1x_1}_{p_1}\ldots U^{a_mx_m}_{p_m}
\]
 for some scalar $\alpha\in\TT$. 
 Hence $\Cst(U_b\mid b\in H)=\Cst(U^{a_i}_{p_i}\mid i\in\{1,\ldots,m\})$. 

 Next, recall that  the action $\gamma'$ is determined by 
\[
\gamma'_\lambda(U_{p_1}^{n_1} U_{p_2}^{n_2} \cdots U_{p_d}^{n_d}) = (\lambda_1^{n_1} \lambda_2^{n_2} \cdots \lambda_d^{n_d}) U_{p_1}^{n_1} U_{p_2}^{n_2} \cdots U_{p_d}^{n_d} \qquad (\lambda \in \TT^d). \]
So let $\gp$ be the restriction of $\gamma'$ to $\Lambda$.  If $E^{\Lambda}\colon \A_\Tht \to \A_\Tht^{\Lambda}$ denotes the conditional expectation obtained by averaging over $\Lambda$, then 
\[
 E^{\Lambda} (U_{p_1}^{n_1} U_{p_2}^{n_2} \cdots U_{p_d}^{n_d}) = \begin{cases}U_{p_1}^{n_1} U_{p_2}^{n_2} \cdots U_{p_d}^{n_d} & \text{if } \lambda_1^{n_1} \lambda_2^{n_2} \cdots \lambda_d^{n_d} =1 \text{ for all } \lambda \in \Lambda,  \\
 0& \text{otherwise.}
 \end{cases}
 \]
Clearly $\lambda_1^{n_1} \lambda_2^{n_2} \cdots \lambda_d^{n_d} =1$ for every $\lambda\in \Lambda$ iff $n_j = 0 \pmod {a_j}$ for $1\leq j \leq m$ and $n_j =0$ for $m+1\leq j \leq d$, that is, iff 
$U_{p_1}^{n_1} U_{p_2}^{n_2} \cdots U_{p_d}^{n_d} = U^{a_1x_1}_{p_1}\ldots U^{a_mx_m}_{p_m}$ for some $x_1, \ldots , x_m \in \ZZ$. This implies that 
$\Cst(U_b\mid b\in H) = \A_\Tht^{\Lambda}$ because $E^{\Lambda}$ is a contraction and products of the form $U_{p_1}^{n_1} U_{p_2}^{n_2} \cdots U_{p_d}^{n_d}$ have dense linear span in $\A_\Tht$.

 
 For the second equality in \eqref{eqn:fpa=C*U=ZAD},
take $b, c\in \ZZ^d$. Then $U_bU_c = e^{2\pi i\braket{ b}{\Thetad c}} U_cU_b$, and hence $U_b \in \mathrm{Z}(\A_{\Tht})$ if and only if~$b\in H$. This shows that $ \A_\Tht^{\Lambda} \subset \mathrm{Z}(\A_\Tht) $ because we have already shown that $\Cst(U_b\mid b\in H) = \A_\Tht^{\Lambda}$. 
In order to establish the reverse inclusion, 
we use the spectral subspaces of  $\gp$. Since $\Lambda$ is a closed subgroup of $\TT^d$, its dual group $\hat \Lambda$  
is the quotient of $\ZZ^d$ by the annihilator of $\Lambda$. If we use the duality pairing of $\TT^d$ with $\ZZ^d$ given by 
$\langle z, a\rangle = z^a = \prod_{j=1}^d z_j ^{a_j}$, then $\Lambda^\perp = \prod_{j=1}^m a_j \ZZ \times \{0\}^{d-m} $ and we have a natural identification $ \hat \Lambda\cong  (\ZZ/{a_1\ZZ})\times \cdots \times(\ZZ/{a_m\ZZ})\times\ZZ^{d-m}  $. Explicitly, denoting by
 $\bar b$ the image of $b\in \ZZ^d$ in $(\ZZ/{a_1\ZZ})\times \cdots \times(\ZZ/{a_m\ZZ})\times\ZZ^{d-m} $,  the corresponding duality pairing between $\Lambda$ and $\hat\Lambda$  is given by $\langle \lambda, \bar b\rangle\coloneqq\lambda^b = \lambda_1^{b_1} \lambda_2^{b_2} \cdots \lambda_d^{b_d}$ for $\lambda\in\Lambda$. The corresponding spectral projection $E^{\Lambda}_{\bar b}$ is given by 
\[
E^{\Lambda}_{\bar b}(x)\coloneqq \int_\Lambda \lambda^{-b} \gp_\lambda(x) \,d\lambda \qquad  (x\in\A_\Tht),
\]
and its range is the $\bar b$\nb-th spectral subspace
\[
E^{\Lambda}_{\bar b} (\A_\Tht)= \clsp \{U_{p_1}^{c_1}\ldots U_{p_m}^{c_m}U_{p_{m+1}}^{b_{m+1}}\ldots U_{p_d}^{b_d}\mid c_i-b_i\in a_i\ZZ\text{ for all } 1\leq i\leq m\}.
\]
Suppose $x\in \mathrm{Z}(\A_\Tht)$. Then $\gp_\lambda(x)\in \mathrm{Z}(\A_\Tht)$ 
for every $\lambda$, and hence also $E^{\Lambda}_{\bar b}(x) \in \mathrm{Z}(\A_\Tht)$. Thus
\[
E^{\Lambda}_{\bar b}(x)U_{p_i}=e^{2\pi i\braket{\sum_{\substack{j=1}}^db_jp_j}{ D p_i}}U_{p_i}E^{\Lambda}_{\bar b}(x)=e^{2\pi i\braket{\sum_{\substack{j=1}}^db_jp_j}{D p_i}}E^{\Lambda}_{\bar b}(x)U_{p_i}
\]
 for every $i = 1, 2, \ldots, d$. We deduce that $E^{\Lambda}_{\bar b}(x)=0$ when  $\bar b\neq 0$. This implies $x=E^{\Lambda} (x)$ (see, for example, \cite[Proposition 17.13]{Exel:Partial_dynamical}),
establishing the inclusion $\mathrm{Z}(\A_\Tht)  \subset E^{\Lambda}(\A_\Tht)= \A_\Tht^{\Lambda}$, and proving the second equality in \eqref{eqn:fpa=C*U=ZAD}. 
 \end{proof}
\end{lemma}

\begin{remark} Notice that the matrix $D'=B^TDB$ from the proof of \lemref{lem:actionandfpa} has the form 
 \begin{equation*}
D'=
\left[
\begin{array}{c|c}
D'_{m\times m} & D'_{m\times (d-m)} \\
\hline
D'_{(d-m)\times m} & D'_{(d-m)\times (d-m)}
\end{array}
\right],
\end{equation*} where $ D'_{(d-m)\times (d-m)}$ is nondegenerate and the remaining entries of $D'$ are integers. Therefore, $\A_\Tht$ is isomorphic to a $d$\nb-dimensional noncommutative torus $\A_{D'}$ whose center $Z(\A_{D'})$ is generated by powers of  the first~$m$ canonical unitary generators of~$\A_{D'}$. The last $d-m$ canonical unitaries generate a simple $\Cst$\nb-subalgebra of~$\A_{D'}$.
\end{remark}

\begin{proposition}\label{pro:tracesfactor}
 Let $E^{\Lambda} \colon\A_\Tht \to \mathrm{Z}(\A_\Tht)$ be the canonical conditional expectation associated to the action $\gp$ from \lemref{lem:actionandfpa}. Then 
the map $\omega \mapsto \omega \circ E^{\Lambda} $ is an affine homeomorphism of the state space of $\mathrm{Z}(\A_\Tht)$
onto the space of tracial states of $\A_\Tht$.
\begin{proof}
Clearly $\omega \mapsto \omega \circ E^{\Lambda} $ is an affine bijection of states of the center to the set of states that factor through $E^{\Lambda}$. The inverse of this map is simply the restriction of states.
We first show every tracial state of $\A_\Tht$ factors through the conditional expectation~$E^{\Lambda}$.
Let~$\tau$ be a trace of~$\A_\Tht$ and let $b\in \ZZ^d$ be such that $b\not\in H$. Let $U_b=U_1^{b_1}U_2^{b_2}\cdots U_d^{b_d}$ and take~$c\in \ZZ^d$ so that $\braket{b}{\Tht c}\not\in\ZZ$. Then
\[
\tau(U_b)=\tau(U_bU_cU_c^*)=\tau(U_c^*U_bU_c)= e^{-2\pi i\braket{b}{\Tht c}}\tau(U_b).
\]
 So we must have $\tau(U_b)=0$ and hence~$\tau = (\tau\restriction_{\mathrm{Z}(\A_\Tht)}) \circ E^{\Lambda}$ since $\Cst(U_b\mid b\in H) = \mathrm{Z}(\A_\Tht)$ by \lemref{lem:actionandfpa}.
 
Next we show that  $\omega\circ E^{\Lambda}$ is indeed a trace of $\A_\Tht$ for every state $\omega$ of $\mathrm{Z}(\A_\Tht)$. 
Recall from \eqref{eqn:VintermsofU} that $\A_\Tht$ is generated by a universal projective unitary representation~$\bar{v}$ of~$\ZZ^d$ with cocycle ~$ \sigma_{\Tht}$. Since 
$\{\bar{v}_b\mid b\in \ZZ^d\}$ has dense linear span, it suffices to show that the commutator of any two of these  elements
lies in the kernel of $E^{\Lambda}$. Let  $b,c \in \ZZ^d$. Then
\[
E^{\Lambda} (\bar{v}_b\bar{v}_c - \bar{v}_c\bar{v}_b) = (\sigma_D(b,c) - \sigma_D(c,b))E^{\Lambda}(\bar{v}_{b+c} ) 
\]
obviously  vanishes when $b+c \notin H$ because then $E^{\Lambda}(\bar{v}_{b+c} ) = 0$.
Assume now that $b+c\in H$. Then
\[
1=e^{-2\pi i\braket{ b+c}{Dc}}=\sigma_D (b+c, c)^2 = \sigma_D (b+c, c) \overline{\sigma_D (c, b+c)}.
\]
Hence $\sigma_D (b+c, c)=\sigma_D (c, b+c)$, which implies $\sigma_D(b,c)=\sigma_D(c,b)$. This shows that $E^{\Lambda} (\bar{v}_b\bar{v}_c - \bar{v}_c\bar{v}_b)$ also vanishes when $b+c \notin H$, and thus $\omega\circ E^{\Lambda}$ is a trace of $\A_\Tht$.
\end{proof}
\end{proposition}

We can now show that the center  of $\A_{\Tht}$ is isomorphic  to a torus of dimension~$m=\operatorname{rank} H$.
This characterisation  can  also be derived from the proof of \cite[Lemma 2.3]{MR731772}.

\begin{proposition}\label{prp:center-iso} 
Let $\Tht\in M_d(\RR)$ be an antisymmetric matrix and let $\A_\Tht$
be the associated noncommutative torus, with canonical unitary generators $U_1, \cdots, U_d$. Let $m=\operatorname{rank}H$ and let $\{p_1,\ldots,p_d\}$ be a basis for $\ZZ^d$ as in \lemref{lem:actionandfpa}. Then there is an isomorphism 
$ \mathrm{C}(\TT^m) \cong \mathrm{Z}(\A_\Tht)$ that sends $z_j$ to $U_{p_j}^{a_j} $ for $j=1,\ldots,m$, where
$z_j\colon \TT^m \to \TT\hookrightarrow\CC$  is the projection onto the $j$\nb-th coordinate.
\begin{proof} Since $\{U_{p_i}^{a_i} \, \mid i=1,2,\ldots,m\}$ is a commuting family of~$m$ unitaries generating $\mathrm{Z}(\A_\Tht)$, there is a canonical surjective homomorphism $\pi_H\colon\mathrm{C}(\TT^m) \to\mathrm{Z}(\A_\Tht)$ such that $\pi_H(z_j) = U_{p_j}^{a_j}$ for $j=1,\ldots, m$.

Now let $a =(a_1, a_2, \ldots, a_m) \in \ZZ^m$
 and consider the action of $\TT^m$ on $\mathrm{C}(\TT^m)$ given by the composite of the translation action with the group homomorphism $\nu\mapsto \nu^a=\prod_{\substack{j=1}}^m\nu_j^{a_j}$. Let $\gamma'$ be the action of $\TT^d$ on $\A_\Tht$ as in \lemref{lem:actionandfpa}. Identify $\TT^m$ with the closed subgroup $\TT^m\times 1_{d-m}$ of $\TT^d$ and notice that for all $\nu,\lambda\in \TT^m$, one has $z_j(\nu^a\lambda) = \nu^{a_j}_j \lambda_j = \nu^{a_j}_jz_j(\lambda)$ and also $\gamma'_{\nu}( U_{p_j}^{a_j}) = \nu^{a_j}_j U_{p_j}^{a_j}$. Thus the homomorphism $\pi_H$ is $\TT^m$\nb-equivariant with respect to the action on $\mathrm{C}(\TT^m)$ described above and the action on $\mathrm{Z}(\A_\Tht)$ obtained by restricting~$\gamma'$ to~$\mathrm{Z}(\A_\Tht)$ and then to~$\TT^m$.
 By \cite[Proposition~2.9]{Exel:Circle_actions}, we deduce that $\pi_H$ is injective, and thus an isomorphism.
  \end{proof}
\end{proposition}

\begin{remark}We believe that the description of the center $\mathrm{Z}(\A_\Tht)$  in \proref{prp:center-iso} and the characterisation of  tracial states of $\A_\Tht$ in  \proref{pro:tracesfactor} are known to experts.  Since we were not able to find an explicit source, we have included precise statements and detailed proofs for completeness and ease of reference.
\end{remark}

\section{Main results}\label{sec:mainresults}
When we combine the results of the preceding two sections we obtain our main theorem, which is a parametrisation of \kmsb states
in terms of states on the center of the noncommutative torus $\A_{\Thetad}$.   By \proref{prp:center-iso},  this center is isomorphic to $\mathrm{C}(\TT^m)$, where $m$ is the degeneracy index of~$\Thetad$. Hence the simplex of \kmsb states is affinely weak* homeomorphic to the space of probability measures on the classical torus $\TT^m$.
\begin{thm} \label{thm:main}
Suppose that $r=(r_1,\ldots,r_k,0, \ldots, 0) \in \RR^{n}$ is a vector with strictly positive first~$k$ coordinates and that
 \[
\Theta=
\left[
\begin{array}{c|c}
\Theta_k &\Lambda \\
\hline
-\Lambda^T &  
\Thetad
\end{array}
\right]
\] 
is an antisymmetric $n\times n$ real matrix with diagonal blocks $\Theta_k$ and $\Thetad$ of sizes $k\times k$ and $d\times d$, respectively. Let $m$ be the degeneracy index of $\Thetad$.
 Then there is an affine weak* homeomorphism of the space $M_1(\TT^m)$ of probability measures on $\TT^m$ onto the space of  \kmsb states of $(\Tt_r(\NN^n,\sigma_\Theta),\alpha^r)$. 
Specifically, if $\{p_1,\ldots,p_m\}$ is a basis for $\ZZ^d$ such that  $\{a_1p_1,\ldots,a_mp_m\}$ is a basis for~$
 H=\{x\in\ZZ^d\mid \braket{x}{\Thetad y}\in\ZZ\text{ for all } y\in\ZZ^d\}
$ as in \lemref{lem:actionandfpa}, then the  affine homeomorphism can be chosen so that the 
extremal \kmsb state $\varphi_{\beta,z}$ associated to the unit point mass at $z\in \TT^m$ is given by
\[
\varphi_{\beta,z} (L_p L_x L_y^* L_q^*) =  \delta_{p,q} \,  {[x-y \in H]}\,  \lambda_{x-y}\, z^c 
\prod_{j=1}^k\frac{e^{-\beta r_j p_j}(1 - e^{-\beta r_j })}{1 - e^{-\beta r_j +2\pi i \braket{\Theta (x-y)}{e_j}}}
\]
where $c=(c_1,\ldots,c_m)$ is the vector of coefficients of $x-y$ with respect to the basis $\{a_1p_1,\ldots,a_mp_m\}$ of~$H$ and $\lambda_{x-y}\in\{-1,1\}$ is such that $\pi(L_xL_y^*)=\lambda_{x-y}U^{c_1}_{a_1p_1}\ldots U_{a_mp_m}^{c_m}$ in $\A_{\Thetad}$.
  \begin{proof} By \lemref{tracesfromcentre}, the homomorphism $\pi\colon \Tt_r(\NN^d,\sigma_{\Thetad})\to \A_{\Thetad}$ that sends the isometry~$L_{e_j}$ to the unitary~$U_j$ yields a weak* homeomorphism between tracial state spaces via $\tau\mapsto\tau\circ\pi$. By  \proref{pro:tracesfactor}, the map that sends a state $\omega$ of the center $\mathrm{Z}(\A_{\Thetad})$ to the tracial state $\omega\circ E^{\Lambda}$ of $\A_{\Thetad}$ is an affine weak* homeomorphism. Combining these maps we conclude that given an isomorphism $\phi\colon\mathrm{Z}(\A_{\Thetad})\xrightarrow{\cong}\mathrm{C}(\TT^m)$, each probability measure~$\mu$ on~$\TT^m$ gives rise to a tracial state $\omega_\mu \circ E^{\Lambda}\circ\pi$ of  $\Tt_r(\NN^d,\sigma_{\Thetad})$, where $$\omega_\mu(b)=\int_{\substack{\TT^m}}\phi(b)\,d\mu\qquad (b\in \mathrm{Z}(\A_{\Thetad})).$$ Such an isomorphism $\phi\colon\mathrm{Z}(\A_{\Thetad})\xrightarrow{\cong}\mathrm{C}(\TT^m)$ always exists by \proref{prp:center-iso}, and thus it follows from~\proref{pro:KMSfromtraces} that the space of probability measures on $\TT^m$ and the \kmsb simplex of $(\Tt_r(\NN^n,\sigma_\Theta),\alpha^r)$ are affinely weak* homeomorphic.

Suppose that the isomorphism $\phi\colon\mathrm{Z}(\A_{\Thetad})\xrightarrow{\cong}\mathrm{C}(\TT^m)$ in question is the one coming from the basis $\{a_1p_1,\ldots,a_mp_m\}$ for~$H$  built in \proref{prp:center-iso}. When we let $\mu$ be the unit point mass at $z\in \TT^m$ and let $\omega_z \circ E^{\Lambda}\circ\pi$ play the role of~$\tau$ in \proref{pro:KMSfromtraces}, we see from \eqref{eqn:Tbeta2} that the weak* homeomorphism described above gives us precisely a extremal \kmsb state $\varphi_{\beta,z}$ given by the formula in the statement of the theorem. That the scalar $\lambda_{x-y}\in\TT$ such that $\pi(L_xL_y^*)=\lambda_{x-y}U_{a_1p_1}^{c_1}U_{a_2p_2}^{c_2}\ldots U_{a_mp_m}^{c_m}$ in $\A_{\Thetad}$ is either $1$ or $-1$ for $x-y\in H$ follows because  $\braket{x}{\Thetad y}=\braket{x-y}{\Thetad y}\in\ZZ$, and as in \eqref{eqn:VintermsofU} we have $$\lambda_{x-y}= e^{\pi i\braket{x}{\Thetad y}}e^{-\pi i \braket{(x-y)-c_1a_1p_1}{\Thetad c_1a_1p_1 }}\ldots e^{-\pi i \braket{c_ma_mp_m}{\Thetad c_{m-1}a_{m-1}p_{m-1}}}.\qedhere$$

\end{proof}
\end{thm}

As  consequences of \thmref{thm:main} for particular choices of~$\Theta$ we get the following two corollaries.
The first one addresses the question of when these systems have unique equilibrium at each $\beta$.

\begin{corollary}\label{cor:nondegenthetad}
If $r>0$ or if $\Thetad$ is nondegenerate, then the system $(\Tt_r(\NN^n,\sigma_\Theta),\alpha^r)$ has a unique \kmsb state for each $\beta>0$.
\end{corollary}

At the other extreme, the second one gives conditions under which the phase transition is maximal, recovering the phase transition result for the building blocks from \cite{AaHRS}.
\begin{corollary}
If $\Thetad \in M_d(\ZZ)$, then $m = d$ and the simplex of  \kmsb states is parametrised by the probability measures on the $d$\nb-torus.
\end{corollary}


\section{Equilibrium at $\beta=\infty$ and $\beta =0$}\label{sec:kmsiandkmso}
In this section we study equilibrium at the extremal inverse temperature values $\infty$ and $0$ for the system $(\Tt_r(\NN^n,\sigma_\Theta), \alpha^r)$.
As before, $\Theta$ is an antisymmetric $n\times n$ real matrix  and   $r=(r_1, \ldots, r_k, 0_d)$ is a vector in $[0,\infty)^n$ with strictly positive first $k$ coordinates.

\subsection{Ground states and \kmsi states}
A state $\varphi $ of $\Tt_r(\NN^n,\sigma_\Theta)$ is a \textit{ground state} for $\alpha$ if and only if the function $z\mapsto \varphi( A\alpha_z(B))$ is bounded on the upper half plane for all analytic elements $A, B\in \Tt_r(\NN^n,\sigma_\Theta)$. See, for example,  \cite[Proposition 5.3.19]{bra-rob} and \cite[Proposition 8.12.3]{ped}. Although initially ground states were also called,
 indistinctly, $\infty$\nb-KMS states, relatively recent results on phase transitions at $\beta=\infty$ related to Bost--Connes systems have required that a distinction be made between generic ground states  as defined above and those that can be obtained as limits of \kmsb states for $\beta \to \infty$. See \cite[Definition 3.7]{CM2008} for the definitions and \cite[Theorem 7.1]{LR2010} for an example in which the two concepts differ. As a result, it has become customary to reserve the terminology "\kmsi states"  for the aforementioned limits.
 
 We wish to compute first the ground states and the \kmsi states of our systems, and we begin with the following characterisation of ground states of $(\Tt_r(\NN^n,\sigma_\Theta), \alpha)$.
\begin{proposition}\label{prop:KMS-characterisationgr} 
Let $n = k+d$ with $k,d\in\NN$ and let $r=(r_1,\ldots, r_k, 0_d)$ with $r_j >0$ for $j \leq k$. Let $\alpha$ be the dynamics determined by~$r$ and suppose that $\varphi$ is a state of $\Tt_r(\NN^n,\sigma_{\Theta})$.
Then $\varphi$ is a ground state of $(\Tt_r(\NN^n,\sigma_\Theta),\alpha)$ if and only if for all~$p,q\in\NN^k\times 0_d$ and $x,y \in 0_k \times \NN^d$, one has
\begin{equation}\label{equ:charac-ground}
\varphi(L_pL_xL_y^*L_q^*)= 0 \text{ unless } p=q=0.
\end{equation}
\begin{proof}  Suppose first that $\varphi$ is a ground state and let $p,q\in\NN^k\times 0_d$ and $x,y \in 0_k \times \NN^d$. Then the function  
\[
z \mapsto \varphi(L_pL_xL_y^*\alpha_z(L_q^*)) = e^{-i \braket{q}{r}z} \varphi(L_pL_xL_y^*L_q^*) 
\]
is bounded on the upper half plane, and so either $q=0$ or $\varphi(L_pL_xL_y^*L_q^*)= 0$. Taking adjoints shows that either $p=0$ or  $\varphi(L_pL_xL_y^*L_q^*) = 0$. This proves that ground states of $(\Tt_r(\NN^n,\sigma_\Theta),\alpha)$ satisfy \eqref{equ:charac-ground}.

Suppose now that $\varphi$ is a state such that \eqref{equ:charac-ground} holds for all  $p,q\in\NN^k\times 0_d$ and $x,y \in 0_k \times \NN^d$. In order to conclude that $\varphi$ is a ground state, it suffices to show that the function $\varphi(X\sigma_z(Y))$ is bounded on the upper half plane
for every $X= L_pL_xL_y^*L_q^*$ and $Y= L_aL_sL_t^*L_b^*$ with $p,q,a,b\in\NN^k\times 0_d$ and $x,y,s,t\in 0_k\times\NN^d$, because the elements of this type are analytic, $\alpha$\nb-invariant, and have a dense linear span in $\Tt_r(\NN^n,\sigma_\Theta)$. Using \defref{dfn-covariance}(3)  we see that $XY=\lambda L_{p-q+q\vee a}L_xL_y^*L_sL_t^*L_{b-a+q\vee a}^*$ for some $\lambda \in\TT$, so \eqref{equ:charac-ground} yields
\[
\varphi(X \alpha_z(Y) ) = e^{i \braket{a-b}{r} z} \varphi(XY )=\begin{cases}e^{i \braket{q\vee a}{r} z}\lambda \varphi(L_xL_y^*L_sL_t^*)&\text{if }q-p=a-b=q\vee a,\\
0  &\text{otherwise}.
\end{cases}
\]
Since $\braket{q\vee a}{r}\geq 0$ the function  $z\mapsto \varphi(X \alpha_z(Y) )$ is bounded on the upper half plane.This shows that $\varphi$ is a ground state for~$(\Tt_r(\NN^n,\sigma_\Theta),\alpha)$ and completes the proof of the proposition.
 \end{proof}
\end{proposition}
 Next we see that ground states  are supported on the corner $Q\Tt_r(\NN^n,\sigma_\Theta)Q$.
\begin{lemma}\label{lem:gound Q}
 Let  $Q $ be the projection from \eqref{eqn:projQ} and suppose that $\varphi$ is a ground state of  $(\Tt_r(\NN^n,\sigma_\Theta), \alpha)$. Then  
 \[
 \varphi(L_xL_y^*)=\varphi(QL_xL_y^*Q), \qquad (x,y\in0_k\times\NN^d).
 \]
\end{lemma}
\begin{proof} Let $\varphi$ be a ground state of  $(\Tt_r(\NN^n,\sigma_\Theta),\alpha)$. Using the expansion formula for~$Q$ given in \eqref{equ:Q-sum} we have
\begin{equation}\label{komaki}
\begin{aligned}
\varphi(QL_xL_y^*Q)=\varphi(L_xL_y^*Q) =\sum_{J\subseteq \{1,\dots, k\}}(-1)^{|J|}\varphi\big(L_xL_y^*\prod_{j \in J}(L_{e_j}L_{e_j}^*)\big).
\end{aligned}
\end{equation} Since $\varphi$ is a ground state, it follows from \proref{prop:KMS-characterisationgr} that $\varphi(L_pL_xL_y^*L_q^*)=0$ whenever $p,q\in\NN^k\times 0_d$ are not both zero. So the terms for $J\neq \emptyset$ in the sum on the right hand side of~\eqref{komaki} will vanish. Hence 
$\varphi(QL_xL_y^*Q)=\varphi(L_xL_y^*)$ as wished.
\end{proof}

\begin{proposition}\label{isomorphisms kmsinfinity states} Let $Q $ be the projection from \eqref{eqn:projQ}. For each state $\omega$ of the corner  $Q \Tt_r(\NN^n,\sigma_\Theta) Q$,  define
\[
T_\infty (\omega) (X)\coloneqq \omega(Q X Q), \qquad (X\in \Tt_r(\NN^n,\sigma_\Theta)).
\] 
Then
\begin{enumerate}
\item  $T_ \infty$ is an affine weak* homeomorphism of the state space of $Q \Tt_r(\NN^n,\sigma_\Theta) Q$ onto the ground state space of 
$(\Tt_r(\NN^n,\sigma_\Theta), \alpha)$; 
\item $T_\infty$ restricted to  tracial states is an affine   weak* homeomorphism of  the tracial state space of $Q \Tt_r(\NN^n,\sigma_\Theta) Q$ onto the \kmsi state space of $(\Tt_r(\NN^n,\sigma_\Theta), \alpha)$.
\end{enumerate} 
\begin{proof}
For part (1),  take a state $\omega$ of the corner $Q \Tt_r(\NN^n,\sigma_\Theta) Q$. Clearly $T_\infty(\omega)$ is a state of $\Tt_r(\NN^n,\sigma_\Theta)$. Also, Lemma~\ref{lemma:Q-properties}(1)  implies that for $p,q\in\NN^k\times 0_d$
and $x,y\in 0_k \times \NN^d$
\begin{align*}
T_\infty(\omega)(L_pL_xL_y^*L_q^*)=\omega(QL_pL_xL_y^*L_q^*Q)
\end{align*}
vanishes unless $p =0 = q$.  Hence $T_\infty(\omega)$ satisfies  equation \eqref{equ:charac-ground} of \proref{prop:KMS-characterisationgr} and so it  is a  ground state of $(\Tt_r(\NN^n,\sigma_\Theta), \alpha)$.  This shows that $T_\infty$ maps the state space of the corner $Q\Tt_r(\NN^n,\sigma_\Theta)Q$ into the ground state space of  $(\Tt_r(\NN^n,\sigma_\Theta), \alpha)$. Clearly $T_\infty$ is a continuous affine map and it is also injective because $\omega$ can be recovered from $T_\infty(\omega)$ by restricting it to the corner.

In order to prove that $T_\infty$ is surjective, suppose $\varphi$  is a ground state of $(\Tt_r(\NN^n,\sigma_\Theta), \alpha)$. Let $\omega_\varphi$ be the restriction of ~$\varphi$ to the corner $Q\Tt_r(\NN^n,\sigma_\Theta)Q$, that is,
\[
\omega_\varphi(X)\coloneqq \varphi(X)\qquad (X\in Q\Tt_r(\NN^n,\sigma_\Theta)Q).
\]
It follows from \lemref{lem:gound Q} that $\omega_\varphi(Q) = \varphi(Q) =1$, and so $\omega_\varphi$  is a state of $  Q\Tt_r(\NN^n,\sigma_\Theta)Q $. We claim that $T_\infty(\omega_\varphi)=\varphi$. To see this,  it suffices to show that $T_\infty(\omega_\varphi)(X)=\varphi(X)$ for $X= L_pL_xL_y^*L_q^*$, where $p,q\in\NN^k\times 0_d$ and $x,y\in 0_k\times\NN^d$. Since $\varphi$ and $T_\infty(\omega_\varphi)$ are ground states,  $T_\infty(\omega_\varphi)(X)=\varphi(X)=0$ when $p$ and $q$ are not both zero by \proref{prop:KMS-characterisationgr}.  Suppose now $p=q=0$. Then 
$$\varphi(L_xL_y^*)=\varphi(QL_xL_y^*Q)=\omega_\varphi(QL_xL_y^*Q)=T_\infty(\omega_\varphi)(L_xL_y^*),$$ where the first equality  follows from \lemref{lem:gound Q}. This shows that $\varphi=T_\infty(\omega_\varphi)$, proving that~$T_\infty$ is surjective. Therefore $T_\infty$ is an affine weak* homeomorphism as asserted, finishing the proof of part (1).

For part~(2) it suffices to verify that $T_\infty$ maps the tracial state space of the corner $Q \Tt_r(\NN^n,\sigma_\Theta) Q$ onto the space of KMS$_\infty$ states of $(\Tt_r(\NN^n,\sigma_\Theta), \alpha)$. Suppose first  $\omega$ is a tracial state of the corner $Q \Tt_r(\NN^n,\sigma_\Theta) Q$ and let $(\beta_j)_{j\in\NN}$ be any sequence of positive real numbers such that $\lim_{\substack{j\to \infty}}\beta_j= \infty$.  By \proref{isomorphisms kmsbeta states} each $T_{\beta_j}(\omega)$ is a KMS$_{\beta_j}$ state for $\alpha$. Since $QL_p =0$ for $p \in \NN^k\times 0_d \setminus \{0\} $ by \lemref{lemma:Q-properties}(1), for every $X \in Q\Tt_r(\NN^n,\sigma_\Theta)Q$ we have 
\[
T_{\beta_j}(\omega)(X) =\frac{1}{Z(\beta_j)} \sum_{p\in \NN^k\times 0_d} e^{-\beta_j \braket{p}{r}} \omega( QL_p^* (QXQ) L_p Q) = \frac{1}{Z(\beta_j)} \omega(X).
\]
Since $T_{\beta_j}(\omega)(Q)\inv =Z(\beta_j)  \to 1$ by \lemref{lemma:reconstruction}, see \eqref{eqn:phiofQcomputation}, this shows that $T_\infty(\omega)$ is KMS$_\infty$.

Suppose now that $\varphi$ is a \kmsi state of $(\Tt_r(\NN^n,\sigma_\Theta),\alpha)$. Then its restriction $\omega_\varphi$  to the corner is a limit of tracial states and satisfies $\omega_\varphi(Q) = \lim_j Z(\beta_j)\inv =1$,
 so it is a tracial state itself. Hence  $\varphi = T_\infty(\omega_\varphi)$ lies in the range of the restriction of $T_\infty$ to tracial states. This completes the proof of part (2).
\end{proof}
\end{proposition}

\subsection{Invariant traces and \kmsop states}
When $\beta =0$ the usual \kmsb condition for a $\Cst$-dynamical system $(A,\alpha)$ just says that the state has to be a trace, cf. 
\cite[Definition 5.3.1]{bra-rob}. Nevertheless, in many applications and equivalent formulations it is necessary to add the extra assumption of $\alpha$\nb-invariance. See, for example, \cite[Theorem 5.3.22]{bra-rob}. In fact, this assumption is often included in the definition,  so that
 the \kmso states are the $\alpha$-invariant traces on $A$, cf. \cite[8.12.2]{ped}.

Here we would like to draw a parallel to the distinction of \kmsi states among all ground states by defining \kmsop states to be the  weak* limits of \kmsb states as $\beta \to 0^+$. 
By \cite[Proposition~5.3.25]{bra-rob} we see that  \kmsop states are traces; in fact, they are also \kmso states because $\alpha$\nb-invariance is preserved under weak* limits.

By taking limits in the formula for \kmsb states obtained in \proref{pro:KMSfromtraces}  we can give a complete description of the space of \kmsop states of our system
$(\Tt_r(\NN^n,\sigma_\Theta),\alpha)$.

\begin{proposition} Identify $\Cst(L_x\mid x\in 0_k\times \NN^d) $ with $\Tt_r(\NN^d,\sigma_{\Thetad})$ canonically and let $\varphi$ be a state of $\Tt_r(\NN^n,\sigma_{\Theta})$. Then $\varphi$ is a \kmsop state of $(\Tt_r(\NN^n,\sigma_{\Theta}),\alpha)$ if and only if it is a trace and satisfies $$\varphi(L_pL_xL_y^*L_q^*)=\delta_{p,q}\varphi(L_xL_y^*)\qquad p,q\in \NN^k\times 0_d,\, x,y\in0_k\times\NN^d.$$ As a consequence, the map that sends a KMS$_{0^+}$ state $\varphi$ of $(\Tt_r(\NN^n,\sigma_{\Theta}),\alpha)$ to its restriction to $\Tt_r(\NN^d,\sigma_{\Thetad})$ is an affine weak* homeomorphism from  the space of KMS$_{0^+}$ states of $(\Tt_r(\NN^n,\sigma_{\Theta}),\alpha)$ onto the space of the tracial states $\tau$ of $\Tt_r(\NN^d,\sigma_{\Thetad})$ such that for all $x,y\in 0\times \NN^d \cong \NN^d$
\begin{equation}\label{eq:plus-states}\tau(L_xL_y^*)=0\quad\text{unless}\quad \braket{\Theta (x-y)}{e_j} \in \ZZ \text{ for } 1\leq j \leq n. 
\end{equation}
\begin{proof} Clearly if $\varphi$ is a \kmsop state, then it is a trace and satisfies $\varphi (L_pL_xL_y^*L_q^*)=\delta_{p,q} \varphi(L_xL_y^*)$. In order to prove the converse and the second part in the statement, we begin by considering the  individual ratios from  \eqref{eqn:TbetaEuler}, namely 
\[
\frac{e^{-\beta r_j p_j}(1 - e^{-\beta r_j })}{1- e^{-\beta r_j +2\pi i \braket{\Theta (x-y)}{e_j}}},
\]
 for each $x,y \in 0_k \times \ZZ^d$ and each $j = 1, 2, \ldots ,k$.  As $\beta \to 0,$ the numerator tends to $0$ while the denominator tends to $1 - e^{2\pi i \braket{\Theta (x-y)}{e_j}}$.
Thus the $j$\nb-th ratio tends to zero unless $\braket{\Theta (x-y)}{e_j} \in \ZZ$, in which case the limit is clearly~$ 1$.  Hence
\begin{equation}\label{eqn:limit0+}
\lim_{\beta\to 0^+} \prod_{j=1}^k \frac{e^{-\beta r_j p_j}(1 - e^{-\beta r_j })}{1- e^{-\beta r_j +2\pi i \braket{\Theta (x-y)}{e_j}}}
=
\begin{cases}
1 & \text{ if } \braket{\Theta (x-y)}{e_j} \in \ZZ  \quad 1\leq j \leq k\\
0 & \text{ otherwise.}
\end{cases}
\end{equation}

So let $\varphi$ be a state of $\Tt_r(\NN^n,\sigma_\Theta)$ and suppose that $$\varphi (L_pL_xL_y^*L_q^*)=\delta_{p,q} \varphi(L_xL_y^*) =\delta_{p,q},$$ for all $p,q\in \NN^k\times0_d$, and $x,y\in 0_k\times\NN^d$. Set $\tau\coloneqq\varphi\restriction_{\Tt_r(\NN^d,\sigma_{\Theta_d})}.$ Then $\tau$ is a trace of $\Tt_r(\NN^d,\sigma_{\Theta_d})$, and because it is the restriction of a trace of  $\Tt_r(\NN^n,\sigma_\Theta)$, it satisfies 
 $\tau(L_xL_y^*) =0$ unless $\braket{\Theta(x-y)}{e_j} =0$ for $j=1, \ldots, n$ by \proref{pro:tracesfactor}. Thus $\tau$ satisfies \eqref{eq:plus-states}. Take a sequence $(\beta_n)_{n\in\NN}$ of positive real numbers converging to~$0$. Taking limits in formula \eqref{eqn:TbetaEuler} for $T_{\beta_n}( \tau\circ \rho_Q\inv)$ and using \eqref{eqn:limit0+}, we see that the weak* limit of $(T_{\beta_n}( \tau\circ \rho_Q\inv))_{n\in\NN}$ is precisely $\varphi$. Hence $\varphi$ is \kmsop state as wished. 

Since a \kmsop state $\varphi$ of $(\Tt_r(\NN^n,\sigma_\Theta),\alpha)$ satisfies $$\varphi (L_pL_xL_y^*L_q^*)=\delta_{p,q} \varphi(L_xL_y^*),$$ it is determined by its restriction $\tau$ to $\Tt_r(\NN^d,\sigma_{\Theta_d})$. Also, we deduce from the above that $\varphi$ is the weak* limit of $(T_{\beta_n}( \tau\circ \rho_Q\inv))_{n\in\NN}$, and $\tau$ is a tracial state of $\Tt_r(\NN^d,\sigma_{\Thetad})$ that satisfies \eqref{eq:plus-states}. This gives the second part in the statement and finishes the proof of the proposition.
\end{proof}
\end{proposition}

Recall from the beginning of \secref{sec:characterisation} that the dynamics determined by a vector $r$ on $\Tt_r(\NN^n,\sigma_\Theta)$ is a subgroup of the gauge action, given by $\alpha_t = \gamma_{e^{i(r) t}} $ for $t \in\RR$. With our convention that the  first $k$ coordinates of $r$ are strictly positive and the rest are zero, $e^{i(r)t} = (e^{itr_1}, e^{itr_2}, \ldots , e^{itr_n}) \in \TT^k \times \{1_d\}$.
By continuity, a state of $\Tt_r(\NN^n,\sigma_\Theta)$ is $\alpha^r$ invariant if and only if it is invariant under the action of the closure of $\{(e^{itr_1}, e^{itr_2}, \ldots , e^{itr_k}) \mid t\in \RR\}$ in $\TT^k$, viewed as a compact  subgroup of  $\TT^k \times \{1_d\}$. 

In order to determine the \kmso states, namely the $\alpha$-invariant traces, we consider the group 
\[
H_{n}^r\coloneqq\{w\in\ZZ^n\mid \braket{w}{\Theta v}\in \ZZ\text{ for all }v\in\ZZ^n\}\cap\{w\in\ZZ^n\mid \braket{w}{r}=0\}
\] and its annihilator
\[
\Lambda_n^r \coloneqq \{ z\in \TT^n\mid z^h = {\textstyle \prod_{j=1}^n z_j^{h_j}} =1 \text{ for all } h\in H_n^r\}.
\]
Let $U_1,\ldots, U_n$ be the canonical unitary generators of~$\A_\Theta$. The conditional expectation obtained from averaging with respect to the restriction of the gauge action on $\A_\Theta$ to $\Lambda_n^r$
satisfies 
\begin{equation} \label{eqn:compositeexpectation}
E^{\Lambda_n^r}(U_v) = \begin{cases} U_v&\text{ if } v \in H_n^r\\
						0 & \text{ otherwise,}
\end{cases}
\end{equation}
where $U_v=U_1^{v_1}\ldots U_n^{v_n}$ with $v=(v_1,\ldots, v_n)\in\ZZ^n$.
\begin{proposition} 
 The map $\omega \mapsto \omega\circ E^{\Lambda_n^r} \circ \pi$ is an affine weak* homeomorphism
of the state space of the fixed point algebra of the action of $\Lambda_n^r$ (which is the $\alpha$-invariant part of ${\mathrm Z}(\A_\Theta)$) and the space of \kmso states of  $(\Tt_r(\NN^n,\sigma_\Theta), \alpha)$.
If $H_n^r$  has nonzero intersection with $\ZZ^k\times 0_d$, then the space of \kmsop states 
is properly contained in the space of \kmso states.
\begin{proof} 
Notice that the range of the conditional expectation $E^{\Lambda_n^r}$ on $\A_\Theta$ is simply 
\[
\clsp \{U_v \mid v \in H_n^r\} = \clsp \{U_v \mid \Theta v\in \ZZ^n \} \cap \clsp \{U_w\mid \braket{w}{r} =0 \}.
\]
 In fact, one can easily verify using \eqref{eqn:compositeexpectation} that $E^{\Lambda_n^r}$ is the composition of
the conditional expectation onto ${\mathrm Z}(\A_\Theta)$  with the conditional expectation onto the fixed point algebra of the dynamics on $\A_\Theta$ corresponding to~$\alpha$.
So by \lemref{tracesfromcentre} and \proref{pro:tracesfactor}, $\omega\circ E^{\Lambda_n^r} \circ \pi$ is a tracial state of $\Tt_r(\NN^n,\sigma_\Theta)$ for every state $\omega$ on the fixed point algebra of $\Lambda_n^r$. Since $E^{\Lambda_n^r}(U_v) =0$ unless $\braket{v}{r} =0$, it is clear that $\omega\circ E^{\Lambda_n^r}\circ \pi$ is  $\alpha$-invariant. To show that the map $\omega \mapsto \omega\circ E^{\Lambda_n^r} \circ \pi$ is also surjective, let $\omega$ be a tracial state of $\A_\Theta$ such that the corresponding tracial state $\omega\circ\pi$ of $\Tt_r(\NN^n,\sigma_\Theta)$ is $\alpha$\nb-invariant and let $v,w\in\NN^n$. Then for all $t\in\RR$ one has $$(\omega\circ\pi)(L_vL_w^*)=e^{it\braket{v-w}{r}}(\omega\circ\pi)(L_vL_w^*),$$ which implies $(\omega\circ\pi)(L_vL_w^*)=0$ if $\braket{v-w}{r}\neq 0$. By \proref{pro:tracesfactor}, we also have $(\omega\circ\pi)(L_vL_w^*)=0$ if $\Theta(v-w)\not\in\ZZ^n$. Hence $\omega\circ\pi=\omega\circ E^{\Lambda_n^r}\circ\pi$ and we then deduce that the map $\omega\mapsto \omega\circ E^{\Lambda_n^r}\circ\pi$ is an affine weak* homeomorphism from the space of states of the fixed point algebra of $\Lambda_n^r$ onto the space of $\alpha$\nb-invariant tracial states of $\Tt_r(\NN^n,\sigma_\Theta)$.
 
For the last assertion, suppose the group $H_n^r$ is not contained in $0_k\times \ZZ^d$. Then we can find $v,w\in\NN^n$ such that $v-w\in H_n^r$ and the first $k$ coordinates of $v-w$ are not all zero. Write $v=p+x$ with $p\in\NN^k\times0_d$ and $x\in 0_k\times\NN^d$ and $w=q+y$ with $q\in\NN^k\times 0_d$ and $y\in 0_k\times\NN^d$. It follows that $p-q\neq 0$, and  $\pi(L_{p}L_xL_y^*L_q^*)$ is an element of the center $\mathrm{Z}(\A_\Theta)$ that is fixed by $\alpha$, so that $\pi(L_{p}L_xL_y^*L_q^*)=(E^{\Lambda_n^r}\circ\pi)(L_{p}L_xL_y^*L_q^*)$. Let $\omega$ be a state of the range of $E^{\Lambda_n^r}$ that does not vanish on $\pi(L_{p}L_x L_y^*L_q^*)$. Then $\omega\circ E^{\Lambda_n^r}\circ\pi$ is an $\alpha$\nb-invariant tracial state of $\Tt_r(\NN^n,\sigma_\Theta)$ that does not vanish on $L_{p}L_xL_y^*L_q^*$. This gives a tracial state of $\Tt_r(\NN^n,\sigma_\Theta)$ that is not a \kmsop state because $p\neq q$.
\end{proof}
\end{proposition}

\begin{exm}
Let $\theta$ be an irrational number in $(0,1/2)$ and consider the matrix 
\begin{equation}\label{eqn:examplematrix}
\Theta=
\left[
\begin{array}{cccc}
0 & 0 &0 & 0\\
0 & 0 &0 &0 \\
0& 0&0 &\theta \\
0 & 0 &-\theta &0
\end{array}
\right].
\end{equation}
Fix $a>0$ and let $\alpha$ be the dynamics on $\Tt_r(\NN^n,\sigma_\Theta)$ defined by $r=(a,a,0,0)$.
Then $k=d=2$ and the lower right corner
$\Theta_d = \left[
\begin{array}{cc}
0 &\theta \\
-\theta &0
\end{array}
\right]
$ has degeneracy index $0$, so that $\A_{\Theta_d} \cong  A_{2\theta}$ is simple and has a unique tracial state. Hence there is a unique \kmsop state. But the center of $\A_\Theta $ is nontrivial and, more importantly, it contains nontrivial $\alpha$\nb-invariant elements, such as $U_1U_2^*$. As a consequence, there are more $\alpha$\nb-invariant traces on $\Tt_r(\NN^n,\sigma_\Theta)$  than just the unique \kmsop state.
\end{exm}

\section{An alternative approach via product systems}\label{sec:alternative}
In this section we realise $\Tt_r(\NN^n,\sigma_\Theta)$ as the Nica--Toeplitz algebra of a compactly aligned product system over $\NN^k$ in which the underlying coefficient algebra is precisely the Toeplitz noncommutative torus $\Tt_r(\NN^d,\sigma_{\Theta_d})$, where $\Theta_d$ denotes the $d\times d$ bottom-right corner submatrix of $\Theta$, and $\sigma_{\Theta_d}$ is the corresponding $2$\nb-cocycle on $\NN^d$. The motivation is to apply the characterisation of \kmsb states of Nica--Toeplitz algebras from \cite{ALN20} and compare it to our parametrisation.

We begin by defining an action of $\NN^k$ by endomorphisms of $\Tt_r(\NN^d,\sigma_{\Theta_d})$, which will then be used to construct a product system over $\NN^k$ with the desired properties. In order to lighten the notation, we let $A\coloneqq \Tt_r(\NN^d,\sigma_{\Theta_d})$ throughout this section. 

\begin{lemma}\label{lem:semi-action} Let $\Theta\in M_n(\RR)$ be an antisymmetric matrix and let $k,d\in\NN$ with $n=k+d$. Let  $\Theta_d$ be the $d\times d$ bottom-right corner submatrix of $\Theta$. Then for each $p\in \NN^k$, there is an automorphism $\rho_p\colon A\to A$ that sends an isometry~$L^{\sigma_{\Theta_d}}_x$ to $\bar{L}^{\sigma_{\Theta_d}}_x\coloneqq \sigma_\Theta(p,x)^2L^{\sigma_{\Theta_d}}_x$ for all $x\in\NN^d$. Moreover, the map $p\mapsto \rho_p$ is an action of $\NN^k$ by automorphisms of~$A$.
\begin{proof} Fix $p\in\NN^k$. As in the statement of the lemma, for $x\in \NN^d$ we let 
$$\bar{L}^{\sigma_{\Theta_d}}_x=\sigma_\Theta(p,x)^2L^{\sigma_{\Theta_d}}_x=e^{-2\pi i \braket{p}{\Theta x}}L^{\sigma_{\Theta_d}}_x\in A.$$
 Because $\sigma_\Theta$ is a bicharacter, the map $x\mapsto \bar{L}^{\sigma_{\Theta_d}}_x$ is an isometric 
 $\sigma_{\Theta_d}$\nb-representation of~$\NN^d$ in~$A$. Such a representation is Nica covariant since $$\bar{L}^{\sigma_{\Theta_d}}_x(\bar{L}^{\sigma_{\Theta_d}}_x)^*=L^{\sigma_{\Theta_d}}_x(L^{\sigma_{\Theta_d}}_x)^*$$ for all $x\in\NN^d$, and so $x\mapsto \bar{L}^{\sigma_{\Theta_d}}_x$ is a covariant isometric $\sigma_{\Theta_d}$\nb-representation. By universal property, we obtain an endomorphism $\rho_p\colon A\to A$ mapping $L^{\sigma_{\Theta_d}}_x$ to $\bar{L}^{\sigma_{\Theta_d}}_x$. Its inverse is the endomorphism induced by the covariant isometric $\sigma_{\Theta_d}$\nb-representation $x\mapsto \sigma_\Theta(x,p)^2L^{\sigma_{\Theta_d}}_x$, and hence $\rho_p$ is in fact an automorphism. Using again that $\sigma_\Theta$ is a bicharacter, we deduce that the map $p\mapsto \rho_p\in\mathrm{Aut}(A)$ is a semigroup action $\rho\colon\NN^k\to\mathrm{Aut}(A)$ on $A$.  This completes the proof of the lemma. 
\end{proof}
\end{lemma} 

Since any automorphism of~$A$ is, in particular, an injective endomorphism with hereditary range, the semigroup homomorphism $\rho\colon \NN^k\to\mathrm{Aut}(A)$ gives rise to a natural product system $A_\rho=(A_{\rho_p})_{p\in\NN^k}$ over $A$ (see, for example, \cite[Example~3.15]{sehnem2021}). We recall here the structure and operations of $A_\rho$. The correspondence $A_{\rho_p}\colon A\leadsto A$ is simply $A$ as a complex vector space. If we write $a\delta_p$ for the element in~$A_{\rho_p}$ that is the image of $a\in A$ under the canonical identification, then the $A$\nb-valued inner product on $A_{\rho_p}$ is given by $\braket{a\delta_p}{b\delta_p}=\rho_p^{-1}(a^*b)$. The right action of~$A$ on~$A_{\rho_p}$ is implemented by $\rho_p$, that is, $a\delta_p\cdot b=a\rho_p(b)\delta_p$ for all $a, b\in A$, while the left action on $A_{\rho_p}$ is simply multiplication of elements in $A$, so $b\cdot a\delta_p=(ba)\delta_p$. The multiplication map $\mu_{p,q}\colon A_{\rho_p}\otimes_AA_{\rho_q}\to A_{\rho_{p+q}}$ is given on an elementary tensor $a\delta_p\otimes b\delta_q$ by $\mu_{p,q}(a\delta_p\otimes b\delta_q)=a\rho_p(b)\delta_{pq}$. With these operations, $A_{\rho}=(A_{\rho_p})_{p\in\NN^k}$ is a product system with coefficient algebra~$A$. 

\begin{lemma} Let $\rho\colon \NN^k\to\mathrm{Aut}(A)$ be the semigroup action from Lemma~\textup{\ref{lem:semi-action}} and let $A_{\rho}=(A_{\rho_p})_{p\in\NN^k}$ be the associated product system over~$\NN^k$ as explained above. Then $A_\rho$ is compactly aligned.
\begin{proof} In order to see that $A_\rho$ is compactly aligned, notice that because $\rho_p$ is an automorphism for all $p\in\NN^k$, each $A_{\rho_p}$ is an imprimitivity $A$\nb-bimodule with left $A$\nb-valued inner product given by $\BRAKET{a\delta_p}{b\delta_p}=ab^*$. Thus for all $p,q\in\NN^k$ one has $$\BRAKET{A_{\rho_p}}{A_{\rho_p}}\BRAKET{A_{\rho_q}}{A_{\rho_q}}=A=\BRAKET{A_{\rho_{p\vee q}}}{A_{\rho_{p\vee q}}}.$$ This implies that $A_\rho$ is compactly aligned (see also Definition 3.7 and Remark 3.8 of~\cite{sehnem2021}).
\end{proof}
\end{lemma}

 Let $A_\rho^\sigma=(A_{\rho_p}^\sigma)_{p\in\NN^k}$ be the product system obtained from $A_\rho$ by twisting the multiplication maps with $\sigma_\Theta$. Specifically, for all $p\in\NN^k$, $A_{\rho_p}^\sigma=A_{\rho_p}$ as an $A$\nb-correspondence, and the isomorphism $\mu^\sigma_{p,q}\colon A_{\rho_p}^\sigma\otimes_A A_{\rho_q}^\sigma\to A_{\rho_{p+q}}^\sigma$ is given by $\sigma_\Theta(p,q)\mu_{p,q}$, for $p,q\in\NN^k$. Notice that the multiplication maps are indeed associative because $\mu$ is and $\sigma_\Theta$ is a $2$\nb-cocycle, and $A^\sigma_\rho$ is compactly aligned because $A_\rho$ is so.

\begin{proposition}\label{pro:nica-coefficient} Let $A_\rho^\sigma=(A_{\rho_p}^\sigma)_{p\in\NN^k}$ be the product system as above.  Let $\iota\colon A\to\Tt_r(\NN^n,\sigma_\Theta)$ be the embedding obtained from the canonical isomorphism $A\cong \Cst(L_x\colon x\in 0_k\times\NN^d)$ from Lemma \textup{\ref{lem:iso-corner}} and identify $\NN^k$ with $\NN^k\times 0_d$ canonically. Then the map that sends $a\delta_p\in A_{\rho_p}^\sigma$ to $\iota(a)L_p\in \Tt_r(\NN^n,\sigma_\Theta)$ induces an isomorphism $\mathcal{N}\mathcal{T}_{A_{\rho}^\sigma}\cong \Tt_r(\NN^n,\sigma_\Theta)$.
\begin{proof} For each $p\in \NN^k$, let $\psi_p$ denote the map that sends $a\delta_p\in A_{\rho_p}^\sigma$ to $\iota(a)L_p\in A$. We will prove that $\psi=\{\psi_p\}_{p\in \NN^k}$ is a Nica covariant representation of $A_\rho^\sigma$. We first show that for all $a\in A$ and $p\in\NN^k$, we have $$L_p\iota(a)L_p^*=\iota(\rho_p(a))L_pL_p^*.$$ Indeed, it suffices to establish this for $a=L_x^{\sigma_{\Theta_d}}(L_y^{\sigma_{\Theta_d}})^*$ with $x,y\in\NN^d$. We compute 
$$L_p\iota(a)L_p^*=L_pL_xL_y^*L_p^*=\sigma_\Theta(p,x)^2\sigma_\Theta(y,p)^2L_xL_y^*L_pL_p^*=\iota(\rho_p(a))L_pL_p^*.$$ Now in order to see that $\psi$ preserves the multiplication in $A_{\rho}^\sigma$, we compute for $a,b\in A$ and $p,q\in\NN^d$, \begin{equation*}
\begin{aligned}
\psi_p(a\delta_p)\psi_q(b\delta_q)=\iota(a)L_p\iota(b)L_q&=\iota(a)L_p\iota(b)L_p^*L_pL_q\\&=\iota(a)\iota(\rho_p(b))L_p(L_p^*L_p)L_q\\&=\sigma_\Theta(p,q)\iota(a\rho_p(b))L_{p+q}\\&=\psi_{p+q}(\mu^\sigma_{p,q}(a\delta_p\otimes b\delta_q)),
\end{aligned}
\end{equation*} proving that $\psi=\{\psi_p\}_{p\in\NN^k}$ is multiplicative. Finally, 
\begin{equation*}
\begin{aligned}
\psi_p(a\delta_p)^*\psi_p(b\delta_p)&=L_p^*\iota(a^*b)L_p=L_p^*\iota(\rho_p(\rho_{p^{-1}}(a^*b)))L_p\\&=L_p^*\iota(\rho_p(\rho^{-1}_{p}(a^*b)))L_pL_p^*L_p\\&=L_p^*L_p\iota(\rho_p^{-1}(a^*b))L_p^*L_p\\&=\iota(\rho_p^{-1}(a^*b))=\iota(\braket{a\delta_p}{b\delta_p}).
\end{aligned}
\end{equation*}
 This shows that $\psi$ also preserves the inner product of $A_{\rho_p}^\sigma$ for each $p\in \NN^k$, and so is a representation of $A_\rho^\sigma$. That $\psi$ is Nica covariant follows because $L_pL_p^*L_qL_q^*=L_{p\vee q}L_{p\vee q}^*$ in $\Tt_r(\NN^n,\sigma_\Theta)$ for all $p,q\in \NN^k$. Thus $\psi=\{\psi_p\}_{p\in\NN^k}$ induces a  homomorphism $\tilde{\psi}\colon \mathcal{N}\mathcal{T}_{A_{\rho}^\sigma}\to \Tt_r(\NN^n,\sigma_\Theta)$ by universal property, which is clearly surjective.

In order to produce the inverse of $\tilde{\psi}$, we will apply the universal property of $\Tt_r(\NN^n,\sigma_\Theta)$. Let $j\in \{1,\ldots, n\}$. If $j\leq k$, we set $\bar{L}_{e_j}\coloneqq \delta_{e_j}\in \mathcal{N}\mathcal{T}_{A_{\rho}^\sigma}$, and for $j>k$, we let $\bar{L}_{e_j}\coloneqq L_{e_{j-k}}^{\sigma_{\Theta_d}}$. Notice that we have introduced no new notation for the image of $A^\sigma_{\rho}$ in $ \mathcal{N}\mathcal{T}_{A_{\rho}^\sigma}$ under the canonical representation. We will show that the set of isometries $\{\bar{L}_{e_j}\mid 1\leq j\leq n\}$ satisfies the relations \eqref{eq:ext-formula} from Proposition~\ref{pro:presentationtoeplitztori}. Indeed, if $i,j\leq k$, then the multiplication in $A_{\rho}^\sigma$ yields $$\bar{L}_{e_i}\bar{L}_{e_j}=\sigma_{\Theta}(e_i, e_j)^2\bar{L}_{e_j}\bar{L}_{e_i}=e^{-2\pi i \theta_{i,j}}\bar{L}_{e_j}\bar{L}_{e_i}.$$ The same will be true for $i,j>k$ because $\sigma_{\Theta_d}$ is simply the restriction of $\sigma_\Theta$ to $(0_k\times \NN^d)\times (0_k\times \NN^d)$. Now for $i\leq k  <j$ or $j\leq k<i$, the relation $\bar{L}_{e_i}\bar{L}_{e_j}=\sigma_{\Theta}(e_i, e_j)^2\bar{L}_{e_j}\bar{L}_{e_i}$ holds because of the right action of $A$ on $A_{\rho_{e_l}}$ for each $1\leq l\leq k$. One can similarly prove that $$\bar{L}^*_{e_i}\bar{L}_{e_j}=e^{2\pi i \theta_{i,j}}\bar{L}_{e_j}\bar{L}_{e_i}^*$$ for all $i,j\in\{1,\ldots, n\}$. By Proposition~\ref{pro:presentationtoeplitztori}, there is a  homomorphism $\varphi\colon \Tt_r(\NN^n,\sigma_\Theta)\to  \mathcal{N}\mathcal{T}_{A_{\rho}^\sigma}$ mapping $L_{e_j}$ to $\bar{L}_{e_j}$ for $j=1,\ldots, n$. This is precisely the inverse of $\psi$.
\end{proof}
\end{proposition}

The dynamics $\bar{\alpha}$ on $\mathcal{N}\mathcal{T}_{A_{\rho}^\sigma}$ induced by the dynamics $\alpha$ on $\Tt_r(\NN^n,\sigma_\Theta)$ via the isomorphism from Proposition~\ref{pro:nica-coefficient} satisfies $$\bar{\alpha}_t(a\delta_p)=N(p)^{it}a\delta_p,\qquad a\in A, \  \ p\in \NN^k.$$ Thus  $\bar{\alpha}$ has the form specified in equation  (1.7) of~\cite{ALN20},
 for the  homomorphism $N\colon\NN^k\to(0,+\infty)$ given by $p\mapsto e^{\braket{p}{r}}$, where $r=(r_1,\ldots, r_k,0_d)$ with $r_j>0$ for $j\leq k$ and $\braket{p}{r}$ is defined by identifying $\NN^k$ with $\NN^k\times 0_d$. Every \kmsb state of $(\mathcal{N}\mathcal{T}_{A_{\rho}^\sigma}, \bar{\alpha})$ for $\beta>0$ is gauge-invariant by Proposition~\ref{prop:KMS-characterisation}. 

	Thus  \cite[Corollary~3.3]{ALN20} shows that  the map $\phi\mapsto \phi\restriction_A$ is a one-to-one correspondence from \kmsb states of $(\mathcal{N}\mathcal{T}_{A_{\rho}^\sigma}, \bar{\alpha})$ to tracial states of the coefficient algebra~$A$
	that satisfy the inequality (3.2) of \cite{ALN20}. It follows from \cite[Remark~3.2]{ALN20} that a tracial state $\tau$ of $A$ satisfies such an inequality if and only if it has the form $$\tau(a)=\sum_{\substack{m\in \NN^k}}e^{-\beta\braket{m}{r}}\tau_0(\braket{\delta_m}{a \delta_m})\qquad (a\in A),$$ for some trace $\tau_0$ of $A$, which is uniquely determined by~$\tau$ and satisfies $\tau_0(1)=Z(\beta)^{-1}$. Note that condition  (3.3) of \cite{ALN20} is automatically satisfied here. The \kmsb state $\phi$ of $(\mathcal{N}\mathcal{T}_{A_{\rho}^\sigma}, \bar{\alpha})$ extending the tracial state $\tau$ of $A$ is then given on an element $b\in \mathcal{N}\mathcal{T}_{A_{\rho}^\sigma}$ by the formula $$\phi(b)=\sum_{\substack{m\in\NN^k}}e^{-\beta\braket{m}{r}}\tau_0(\braket{\delta_m}{b\delta_m}) ,$$ where the left action of $\mathcal{N}\mathcal{T}_{A_{\rho}^\sigma}$ on $A^\sigma_{\rho_m}$ is obtained by restricting the left action of  $\mathcal{N}\mathcal{T}_{A_{\rho}^\sigma}$ on the Fock space $\mathcal{F}(A_\rho^\sigma)=\bigoplus_{\substack{p\in \NN^k}}A_{\rho_p}^\sigma$ of $A_\rho^\sigma$ to the direct summand $A_{\rho_m}^\sigma$. Using the terminology of Corollary~6.10(ii) of~\cite{ALN20} this implies that  for $\beta>0$ all \kmsb states of our system are of finite type. Combining this with the isomorphism $\mathcal{N}\mathcal{T}_{A_{\rho}^\sigma}\cong\Tt_r(\NN^n,\sigma_\Theta)$ from \proref{pro:nica-coefficient},  we obtain from \cite{ALN20} a one-to-one correspondence between tracial states of $A = \Tt_r(\NN^d,\sigma_{\Thetad})$ and \kmsb states of $(\Tt_r(\NN^n,\sigma_\Theta), \alpha)$ for which the resulting map from traces to \kmsb states coincides with the map given in \proref{pro:KMSfromtraces}. We emphasize that this application of the results from \cite{ALN20} to describe the phase transition for $\beta >0$ depends on the choice of a product system that is adapted to the dynamics in such a way that  \kmsb states are gauge-invariant, a choice that is guided by our Proposition~\ref{prop:KMS-characterisation}.

\end{document}